\documentclass[11pt, reqno]{amsart}
\usepackage[top=90pt,bottom=90pt,left=90pt,right=90pt]{geometry}
\usepackage[utf8]{inputenc}
\usepackage[T1]{fontenc}
\usepackage[english]{babel}
\usepackage{enumerate}
\usepackage{graphicx}
\usepackage[size=footnotesize]{caption}
\usepackage[size=footnotesize]{subcaption}
\usepackage{fullpage}
\usepackage{bbm,amsmath,amssymb,amsfonts,xcolor,mathtools, verbatim,color,epsf}
\usepackage{bbold}
\usepackage{url}
\usepackage[bookmarksopen=false,breaklinks=true,%
      hyperindex=true,pdfstartview=FitH,colorlinks=true,%
      pdfpagelabels=true,colorlinks=true,linkcolor=black,%
      citecolor=black,urlcolor=black,hypertexnames=false%
      ]
   {hyperref}
\newcommand{\SSigma}{\mathcal{S}\hspace{-0.5mm}\mathit{\Sigma}}
\newcommand{\N}{\mathbb{N}}

\newcommand{\R}{\mathbb{R}}
\renewcommand{\S}{\mathcal{S}}
\newcommand{\C}{\mathbb{C}}
\newcommand{\X}{\underline{X}}

\newcommand{\Q}{\mathbb{Q}}

\renewcommand{\leq}{\leqslant}
\usepackage{faktor}
\usepackage{MnSymbol}

\usepackage{etoolbox}
\makeatletter
\patchcmd{\@setaddresses}{\indent}{\noindent}{}{}
\patchcmd{\@setaddresses}{\indent}{\noindent}{}{}
\patchcmd{\@setaddresses}{\indent}{\noindent}{}{}
\makeatother

\newtheorem{theorem}{Theorem}[section]
\newtheorem{proposition}[theorem]{Proposition}
\newtheorem{corollary}[theorem]{Corollary}
\newtheorem{lemma}[theorem]{Lemma}

\newtheorem{definition}[theorem]{Definition}
\newtheorem{example}[theorem]{Example}
\newtheorem{remark}[theorem]{Remark}

\DeclareMathOperator\tr{tr}

\DeclareFontFamily{U}{mathx}{\hyphenchar\font45}
\DeclareFontShape{U}{mathx}{m}{n}{
      <5> <6> <7> <8> <9> <10>
      <10.95> <12> <14.4> <17.28> <20.74> <24.88>
      mathx10
      }{}
\DeclareSymbolFont{mathx}{U}{mathx}{m}{n}
\DeclareFontSubstitution{U}{mathx}{m}{n}
\DeclareMathAccent{\widecheck}{0}{mathx}{"71}

\title{Any-dimensional Positivstellensätze \\ for symmetric functions}

\author{Sebastian Debus}
\address{Fachbereich Mathematik und Statistik, Universität Konstanz, 78457 Konstanz, Germany}
\email{sebastian.debus@uni-konstanz.de}

\author{Robin Schabert}
\address{Fachbereich Mathematik und Statistik, Universität Konstanz, 78457 Konstanz, Germany}
\email{robin.schabert@uni-konstanz.de}

\date{\scriptsize{\today}}

\begin{document}
\setlength{\parindent}{0pt}

\begin{abstract}
Positivstellensätze provide certificates of positivity for polynomials. Extending these certificates to symmetric functions, uniformly across all dimensions, presents structural challenges. For instance, the underlying domain is not semialgebraic. In this paper, we prove two Positivstellensätze for symmetric functions that are uniformly bounded below by some $\varepsilon > 0$. These are infinite-dimensional analogs of theorems of Pólya and Reznick. The proof relates evaluations of the (truncated) power sum map $(p_2,p_3,\dots)$ to moments of discrete probability measures on the compact interval $[-1,1]$. This yields a characterization of the closure of the orbit space of the infinite symmetric group on the sphere. Finally, we provide an alternative proof of existing Positivstellensätze for normalized symmetric functions.
\end{abstract}

\maketitle

\section{Introduction}
In various areas of mathematics, such as optimization or algebra, one is concerned with the problem of verifying whether a polynomial is non-negative on a specified domain. In real algebraic geometry, there are various classical Positivstellensätze due to Krivine-Stengle \cite{krivine1964anneaux, stengle1974nullstellensatz}, Putinar \cite{putinar1993positive}, Schmüdgen \cite{schmudgen1991}, Pólya \cite{polya} and Reznick \cite{reznick1995uniform}. They have in common that they apply to polynomials which are non-negative or positive on a semialgebraic set. Pólya's and Reznick's results give uniform denominators in Hilbert's 17th problem for positive polynomials. In contrast, non-negative polynomials can have an obstruction to such denominators in the form of bad points \cite{benoist2022bad,delzell1980constructive}.

In this paper, we address the problem of the existence of symmetric any-dimensional global positivity certificates which was recently posed in \cite{levin2025any}. A symmetric function naturally truncates to a family of symmetric problems by setting all but finitely many variables to zero. Thus, non-negative symmetric functions can be viewed as universal inequalities, valid in every dimension. 

Verifying whether a symmetric function is non-negative actually requires the decision whether a polynomial is non-negative on a non-semialgebraic set. Thus, standard techniques from real algebraic geometry do not apply. 
The qualitative relation between homogeneous non-negative symmetric functions and those which are sums of squares was classified in \cite{acevedo2025symmetric} and the equality cases for forms are precisely the same as in Hilbert's general classification for polynomials \cite{hilbert1888darstellung}. 

Beyond symmetric functions, any-dimensional optimization problems also include polynomial inequalities in graph homomorphism densities \cite{lovasz}, normalized symmetric function inequalities \cite{acevedo2025power,blekherman2021symmetric}, and polynomial inequalities in traces of matrices \cite{klep2022optimization,klep2021} and moments of probability measures \cite{klep2026sums}. The normalized symmetric functions coincide with univariate pure trace polynomials \cite{klep2021}, as well as univariate pure moment polynomials \cite{klep2026sums}. While the multivariate extensions of these frameworks are significantly more general \cite{klep2022optimization, klep2026sums}, these works directly imply Stellensätze for normalized symmetric functions.
In \cite{levin2025any} a universal framework for any-dimensional non-negativity problems was introduced in the language of representation stability. 

To the best of our knowledge, we are not aware of existing algebraic Positivstellensätze for symmetric functions. If one restricts the degree of all considered symmetric functions, one must verify non-negativity of a polynomial in a pair $(x,y)$ of tuples of variables, where $x$ lies in a semialgebraic set and $y$ is contained in a discrete set \cite{acevedo2025symmetric}. Analytic Positivstellensätze relying on integral representations with respect to positive measures over universally quantified sets were recently proved in \cite{hu2026}. In our paper, we provide algebraic positivity certificates for symmetric functions.

\subsection*{Main contributions}
We prove two Positivstellensätze for symmetric functions. 
The first Positivstellensatz (Corollary \ref{cor:infinite Polya}) is an analog of Pólya's Positivstellensatz \cite{polya} for symmetric functions. If a homogeneous even symmetric function is strictly bounded below by some $\varepsilon > 0$, we prove that there exists an integer $k$ such that $(\sum_{i=1}^\infty X_i^2)^kf$ is a linear combination of even monomial symmetric functions with non-negative coefficients. 

The second Positivstellensatz (Theorem \ref{thm: dimension-independent reznick}) is a symmetric function analog of Reznick's Positivstellensatz \cite{reznick1995uniform}. Let $p_k = \sum_{i=1}^\infty X_i^k$ denote the $k$-th power sum function. We prove that for every homogeneous symmetric function $f=g(p_2,\ldots,p_{2d})$ whose truncations are uniformly bounded away from zero, there exists an integer $k$ such that $(\sum_{i=1}^\infty X_i^2)^kf$ is a sum of squares. We generalize this result in Theorem \ref{thm: dimension-independent reznick non-homogeneous} to non-homogeneous symmetric functions and also to symmetric functions including $p_1$ in Theorem \ref{thm: dimension-independent reznick with p_1}.

To prove Theorem \ref{thm: dimension-independent reznick} we classify the set of ring homomorphisms $A:=\R[p_2,p_3,\ldots] \to \R$ that are non-negative on the preordering generated by sums of squares and the principal ideal $(1-p_2)$ in $A$. We prove in Corollary \ref{cor:K_T-description} that the ring homomorphisms are essentially the closure of the moments $(m_0,m_1,\ldots)$ of all probability measures on the compact interval $[-1,1]$ of the form $\sum_{i=1}^\infty x_i^2\cdot \delta_{x_i}$ for $x \in \ell^2(\R)$. We also prove that the set of all such moments $(m_0,m_1,\ldots)$ can be seen as the orbit space of the infinite symmetric group of the unit sphere in $\ell^2(\R)$ (compare Proposition \ref{prop: orbit space} and Corollary \ref{cor:K_T-description}).

In Example \ref{ex: bad point reznick} we present a homogeneous symmetric function $f$ which is non-negative but $p_2^k f$ is not a sum of squares for any $k$ for sufficiently many variables. We achieve this by constructing a symmetric function whose truncations have bad points.

Finally, we prove Positivstellensätze for normalized symmetric functions in Section \ref{sec:normalized}. These follow already from the more general work in \cite{klep2022optimization,klep2026sums, klep2021} but our proof is elementary. By reducing the problem to classical Positivstellensätze on semialgebraic sets, the certificates follow directly as a consequence of the structural observations in \cite{acevedo2025power}.

\subsection*{Notation}
We write $\N$ for the set of positive integers and $\N_0$ for the set of non-negative integers. For $x \in \R$ we denote by $\delta_x$ the \emph{Dirac measure} on $x$. For a measure $\mu$ on $\R$ and $x \in \R$ we write $\mu(x)$ instead of $\mu(\{x\})$. For $1 \leq  p < \infty$ we denote by $\ell^p$ the subspace of $\R^\N$ of all sequences $x=(x_n)$ with finite $p$-norm $\Vert x \Vert_p := \left( \sum_n |x_n|^p\right)^{1/p}$. Moreover, we denote by $\R^\infty$ the set of all sequences $x\in \R^{\N}$ where all but finitely many $x_i$ are $0$. We view $\R^n \subset \R^\infty$ under \emph{zero-padding} $x \mapsto (x_1,\ldots,x_n,0,0,\ldots)$ and thus $\R^\infty = \bigcup_{n \in \N} \R^n$. We equip $\R^\N$ with the product topology.

For a commutative ring $A$ we denote by $\sum A^2$ the set of sums of squares in the ring $A$. The ring $\R[[\X]]:=\R[[X_1,X_2,\ldots]]$ denotes the ring of formal power series in countably infinitely many variables $X_1,X_2,\dots$.
A polynomial $f(X_1,\ldots,X_n)$ is \emph{symmetric} if $f(X_1,\ldots,X_n)=f(X_{\sigma(1)},\ldots,X_{\sigma(n)})$ for all permutations $\sigma$ of $[n]$. Analogously, a formal power series is symmetric if it is of bounded degree and invariant under all permutations of the variables. 
We denote the subset of $\R[X_1,\ldots,X_n]$ of symmetric sums of squares by $\SSigma_n$ and the subset of symmetric non-negative polynomials by $\mathcal{SP}_n$. For a positive integer $k$ we write $p_k = \sum_{i=1}^\infty X_i^k$ for the $k$-th \emph{power sum function} and $e_k = \sum_{i_1 < \dots < i_k} X_{i_1}\dots X_{i_k}$ for the $k$-th \emph{elementary symmetric function}.
For a partition $\lambda$ of length $\ell$ we define the \emph{monomial symmetric function} as $m_{\lambda}=\sum X_{i_1}^{\lambda_1}\cdots X_{i_\ell}^{\lambda_\ell}$, where the sum is taken over all distinct monomials that can be obtained by permuting the indices of the variables.
For a symmetric function $f$ we denote by $f^{(n)}$ its truncation to a symmetric polynomial in $n$ variables, e.g., we write $e_k^{(n)}$ and $p_k^{(n)}$ for the truncations of $e_k$ and $p_k$ to $n$ variables.
We say a polynomial $f\in \R[X_1,\dots,X_n]$ is \emph{even} if every monomial of $f$ has even degree in each variable. 

\section{Any-dimensional sos and psd}
In this section, we introduce the objects that we study. We consider symmetric polynomials that are non-negative (positive) in any dimension and relate these to non-negative (positive) symmetric functions. In doing so, we explore truncated moments of certain discrete measures on $\R$.

\begin{definition}\label{def: symmetric any dim sos psd}
    Let $f(X_1,X_2,\ldots)$ be a symmetric function. 
    \begin{enumerate}
        \item Then $f$ is called an \emph{any-dimensional sum of squares} (short \emph{any-dimensional sos}) if $f^{(n)} \in \SSigma_n$
 for all $n\in \N$. We write $\SSigma$ for the set of all any-dimensional sos symmetric functions.
 \item We say that $f$ is \emph{any-dimensional non-negative} (short \emph{any-dimensional psd}) if $f^{(n)} \in \mathcal{SP}_n$ for all $n \in \N$. We denote the set of all any-dimensional psd symmetric functions by $\mathcal{SP}$.
 \item Moreover, if $f$ is homogeneous and $f^{(n)}$ is positive definite for all $n \in \N$, or if $f^{(n)}$ is positive for all $n \in \N$, we call $f$ \emph{any-dimensional positive definite}.
    \end{enumerate}
\end{definition}
In particular, let $f=g(p_1,\ldots,p_{2d})$ be a symmetric function for a polynomial $g \in \R[Z_1,\ldots,Z_{2d}]$. Then $f$ is any-dimensional non-negative (positive definite) if the polynomial $g(Z_1,\ldots,Z_{2d})$ is non-negative (positive) on the set $\bigcup_{n \in \N} (p_1,\ldots,p_{2d})(\R^n \setminus \{0\})$. This set is not semialgebraic by \cite[Corollary~2.25]{acevedo2025wonderful} which makes deciding whether a sequence is contained in $\mathcal{SP}$ potentially highly challenging. It is an open question if deciding containment in $\mathcal{SP}$ is decidable.

Observe that if $f^{(n+1)} \in \SSigma_{n+1}$ then $ f^{(n)} \in \SSigma_{n}$ (and the analogous statement holds with respect to containment in $\mathcal{SP}_{n+1}$ and $\mathcal{SP}_n$) which follows from the identity \[p_d^{(n+1)}(X_1,\ldots,X_n,0)=p_d^{(n)}(X_1,\ldots,X_n).\]
The following example illustrates instances of any-dimensional symmetric sos polynomials and shows that non-negativity does not have to be preserved in an increasing number of variables.
    
\begin{example}
    \begin{enumerate}
        \item Every sum of squares of symmetric functions is an any-dimensional sos.
        \item $p_2^{(n)}$ is not a sum of squares of symmetric polynomials, but $p_2$ is still any-dimensional sos.
        \item Note that $-(p_1^{(2)})^2+2p_2^{(2)}=(x_1-x_2)^2 \in \SSigma_2$. However, already for $n=3$ we have $-(p_1^{(3)})^2+2p_2^{(3)} \not\in \mathcal{SP}_3$. More generally, $f^{(n)}:=-(p_1^{(n)})^2+dp_2^{(n)}
        \in \SSigma_n$ for $n\leq d$ but $f^{(n)} \not\in \mathcal{SP}_n$ for all $n>d$.
    \end{enumerate}
\end{example}

The qualitative relation between homogeneous any-dimensional psd and any-dimensional sos symmetric functions has recently been investigated in \cite{acevedo2025symmetric}. We write $\SSigma_{=2d}$ and $\mathcal{SP}_{=2d}$ for the set of homogeneous elements of degree $2d$ in $\SSigma$ and $\mathcal{SP}$, respectively. In the following theorem we view these sets as subsets of $\R^{\pi(2d)}$ under the identification $\sum_{\lambda \vdash 2d} c_\lambda p_{\lambda_1}\cdots p_{\lambda_\ell} \mapsto (c_\lambda)_{\lambda}$, where $\pi(2d)$ denotes the number of integer partitions of $2d$. 
    \begin{theorem}[\cite{acevedo2025symmetric}~Theorem~3.12]
    For any even degree $2d \geq 4$ the relation $\SSigma_{=2d} \subsetneq \mathcal{SP}_{=2d}$ holds. Moreover, in all these cases the set $\SSigma_{=2d}$ is semialgebraic but $\mathcal{SP}_{=2d}$ is not semialgebraic.
    \end{theorem}
    The theorem actually also characterizes the equality cases for even symmetric homogeneous any-dimensional sos and psd symmetric functions, where the respective sets differ from degree $6$ onward. 

The set of any-dimensional sos sequences is a spectrahedron and symmetry reduction methods can be used to describe those any-dimensional sos of a fixed bounded degree. We refer to \cite{blekherman2021symmetric} for background information. 
\begin{example}[Characterization of degree 2 any-dimensional symmetric sos]
A symmetric polynomial $f^{(n)}$ of degree $2$ in $n$ variables is sos if and only if there are $a,b,c,d \in \R$ with $a,c,d \geq 0$ and $ac-b^2 \geq 0$ with
\begin{align*}
 f^{(n)}&=\tr \left( \begin{pmatrix}
    a & b \\ b & c
\end{pmatrix} \cdot \begin{pmatrix}
    1 & X_1+\dots+X_n \\ X_1+\dots+X_n & (X_1+\dots+X_n)^2
\end{pmatrix} \right)+d\left(p_2-\frac{1}{n}p_1^2\right)\\
&=a+2bp_1+(c-\frac{1}{n}d)p_1^2+dp_2 .\end{align*}
Thus, every any-dimensional sos symmetric function of degree $2$ expressed in power sums must be of the form
\[ \alpha+2\beta p_1+ \gamma p_1^2+ \delta p_2\]
for some $\alpha, \gamma, \delta \geq 0$ and $\alpha \gamma -\beta^2\geq 0$. It follows from Hilbert's theorem from 1888 \cite{hilbert1888darstellung}  that also any any-dimensional non-negative symmetric function of degree $2$ can be expressed in that form.
\end{example}
We refer to Examples A.7., A.8. and A.9. in \cite{acevedo2025symmetric} for further characterizations of homogeneous any-dimensional sos of degree 4, and homogeneous even any-dimensional sos of degrees 6 and 8.

Verifying whether a symmetric function is any-dimensional psd requires bounding the infima of symmetric polynomials across all dimensions. First examples of any-dimensional psd sequences can be found in \cite{acevedo2025symmetric,levin2025any}. We present these examples and make our attempt on certifying any-dimensional non-negativity through uniform denominators in a rational sums of squares certificate visible.
    \begin{example}
        \begin{enumerate}
            \item A homogeneous example of an element in $\mathcal{SP} \setminus \SSigma$ is \[f=4p_1^4-5p_2p_1^2-\frac{139}{20}p_3p_1+4p_2^2+4p_4.\] 
            The polynomial $f^{(n)}$ is not sos for $n \geq 4$ (see \cite[Theorem~3.6]{acevedo2025symmetric}). Numerical computations indicate that $p_2^{(n)}f^{(n)}$ is sos for $4 \leq n \leq 12$.
            \item  A non-homogeneous example of an element in $\mathcal{SP} \setminus \SSigma$ is   \begin{align*}   f = &\tfrac{1}{2}p_6 - \tfrac{15}{16}p_4p_2 + \tfrac{1}{16}p_4p_1^2 + \tfrac{15}{32}p_2^3 - \tfrac{3}{32}p_2^2p_1^2 - \tfrac{1}{32}p_2p_1^4 + \tfrac{1}{32}p_1^6 + \tfrac{3}{8}p_4 - \tfrac{9}{16}p_2^2 + \tfrac{3}{8}p_2p_1^2 - \tfrac{3}{16}p_1^4 + 1.
    \end{align*}
    The polynomial $f^{(n)}$ is sos only for $n=1$ (see \cite[Example~1.5]{levin2025any}).
    For $n\in\{2, 3, 4, 5\}$ numerical computations indicate that $(p_2^{(n)})^kf^{(n)}$ is sos if $k$ is at least $5$.
        \end{enumerate}
        In both cases we do not know if the product remains sos in all dimensions.
    \end{example}

\subsection{Connection between any-dimensional psd and non-negative symmetric functions}

While we introduced an any-dimensional psd or sos symmetric function as a sequence of symmetric polynomials in $\mathcal{SP}_n$ and $\SSigma_n$, respectively, we can equivalently think about non-negative symmetric functions in countably infinitely many variables.
\begin{definition}
The graded ring of \emph{symmetric functions} over $\R$ is the polynomial ring generated by all \emph{power sum symmetric functions} $p_k = \sum_{i=1}^\infty X_i^k$ and we denote it by $\Lambda$. The \emph{$k$-th elementary symmetric function} is $e_k = \sum_{1 \leq i_1 < \dots < i_k} X_{i_1}\dots X_{i_k}$. In particular, we have $\Lambda = \R[p_1,p_2,\ldots] = \R[e_1,e_2,\ldots]$.   
\end{definition}
The ring $\Lambda \subset \R[[\X]]=\R[[X_1,X_2,\ldots]]$ is a subring of the ring of formal power series in countably infinitely many variables. Each element in $\Lambda$ has bounded degree and is invariant with respect to permutation of every subset of coordinates. The truncation map $\Lambda \ni f(\X) \mapsto f(X_1,\ldots,X_n,0,\ldots)$ allows us to identify the sequence $(g(p_1^{(n)},\ldots,p_{2d}^{(n)}))_{n \in \N}$ with $g(p_1,\ldots,p_{2d}) \in \Lambda$. Thus, $f(\X) \in \Lambda$ is any-dimensional psd, or sos respectively, if $f(X_1,\ldots,X_n,0,\ldots) \in \mathcal{SP}_n$, or contained in $\SSigma_n$ respectively, for all $n \in \N$.

\begin{remark}
It is important to distinguish between formal power series and symmetric functions which we view as polynomials in all numbers of variables via truncation.
The standard power series expansion $\sqrt{1+t}=\sum_{k=0}^\infty \binom{1/2}{k}t^k$ shows that
any symmetric function with positive coefficients can be written as an infinite sum of squares of formal power series. 
However, under truncation this does not reduce to a valid sos decomposition for a symmetric polynomial as a sum of squares of polynomials. 
\end{remark}

All symmetric functions can naturally be evaluated at points in $\ell^1$. Recall that $\ell^p \subset \ell^q$ for all $1 \leq p \leq q < \infty$. Thus, for integers $1 \leq k_1 < \dots < k_d$ a symmetric function $f=g(p_{k_1},\dots,p_{k_d}) \in \Lambda$ can also be evaluated on the larger space $\ell^{k_1}$.
\begin{definition}
Let $1 \leq k_1 < k_2 < \dots < k_d$ be positive integers and $f=g(p_{k_1},\ldots,p_{k_d}) \in \Lambda$ a symmetric function. Then $f$ is \emph{positive} (respectively \emph{non-negative}) if $f(x) > 0$ (respectively $f(x)\geq 0$) for all $x \in \ell^{k_1}$. We call $f \in \Lambda$ \emph{sos}, if $f = \sum_{i=1}^\infty g_i^2$ for some $g_i \in \R[[\X]]$ and each $g_i$ has degree at most $\frac{\deg (f)}{2}$. 
\end{definition}
Note that if $f \in \Lambda$ is sos, we obtain that $f^{(n)} \in \SSigma_n$ for all $n \in \N$, because the truncation maps $\Lambda \to \R[X_1,\ldots,X_n],  f(\X) \mapsto f(X_1,\ldots,X_n,0,0,\ldots)$ are ring homomorphisms. Moreover, if $f^{(n)}$ is sos for all $n \in \N$, then the symmetric function is sos by \cite[Proposition~A.6]{acevedo2025symmetric}. Thus, $\SSigma$ is the set of all $f \in \Lambda$ that are sos.

We will now prove that the notions of $f \in \Lambda$ being positive (non-negative) and $f$ being any-dimensional positive definite (psd) are equivalent.

\begin{lemma} \label{lem: images on ell ball versus unions of Rn}
For $d \geq 1$ the sets $(p_1,\ldots,p_d)(\ell^1)$ and $(p_1,\ldots,p_d)\left( \R^{\infty}\right)$ are equal.
\end{lemma}
One can think about the map $(p_1,\ldots,p_d) : \ell^1 \to \R^d$ as computing the moments $(m_1,\ldots,m_d)$ of a discrete (not necessarily finite) measure $\sum_{i=1}^\infty \delta_{x_i}$ on $\R$, where $x=(x_1,x_2,\dots) \in \ell^1$. Thus the lemma can be seen as a variation of the truncated moment problem on the real line and says that if we only consider the moments $(m_1,\ldots,m_d)$, allowing an infinite number of non-zero variables in $\ell^1$ does not create any new combination of moments which we could not already achieve with a finite sequence in $\R^\infty$. Results in this direction are well known. For instance, Tchakaloff's theorem implies that the truncated moments of a measure $\nu$ on $\R^n$ agree with the truncated moments of a finite discrete measure $\mu=\sum_{i=1}^N a_i \delta_{x_i}$ with $a_i > 0$, $\operatorname{supp}(\mu) \subset \operatorname{supp}(\nu)$ with an explicit bound on $N$ (see e.g. \cite[Theorem~5.9]{laurent2008sums}). In our setting, we consider only discrete measures on $\R$ where all weights are $1$.
\begin{proof}[Proof of Lemma \ref{lem: images on ell ball versus unions of Rn}]
Since $\R^\infty \subset \ell^1$ it suffices to prove that for each $x \in \ell^1$ there exists $y \in \R^n$ for some $n$ with the same image under the map $P_d := (p_1,\ldots,p_d) : \ell^1 \to \R^d$. If $x$ contains only finitely many coordinates different from zero this is clear. 
Moreover, the set $P_d\left(\R^\infty \right)$ is closed under addition, since the concatenation of $a \in \R^{n_1}, b \in \R^{n_2}$ as an element in $\R^{n_1+n_2}$ evaluates to $P_d(a)+P_d(b)$.

Suppose that $x \in \ell^1$ has infinitely many non-zero coordinates.
Then $x$ cannot contain infinitely many times the same non-zero entry, as $x \in \ell^1$. Thus, for some sufficiently large $m$ the point $x^{(m)} := (x_1,\ldots,x_m) \in \R^m$ contains at least $d$ distinct coordinates indexed by $i_1< \cdots < i_d$. 
The map \[ \nu_d : \R^d \to \R^d, z \mapsto P_d(x_1,\ldots,x_{i_1-1},z_1,x_{i_1+1},\ldots,x_{i_d-1},z_d,x_{i_d+1},\ldots,x_m)\] 
has the Jacobian matrix
\[ J(z_1,\ldots,z_d) =\begin{pmatrix}
    1 & \dots & 1 \\
    2z_1 & \dots & 2 z_d\\
    3z_1^2 & \dots & 3z_d^2 \\
    \vdots & & \vdots \\
    dz_1^{d-1} & \dots & d z_d^{d-1}
\end{pmatrix}.\]
Clearly $\nu_d(\R^d) \subset P_d(\R^m)$ and since $\det J(x_{i_1},\dots,x_{i_d})$ is equal to the Vandermonde determinant of $(x_{i_1},\ldots,x_{i_d})$ multiplied by the scalar $d!$, the map $\nu_d$ is locally a diffeomorphism at $(x_{i_1},\ldots,x_{i_d})$ by the inverse function theorem. So, there exists $\varepsilon > 0$ with $B_{\varepsilon}(P_d(x^{(m)})) \subset P_d\left( \R^\infty \right)$. Since $ P_d\left(  \R^\infty \right)$ is closed under addition, we can conclude 
\[ B_{\varepsilon}(P_d(x^{(m)})+ \omega ) \subset (p_1,\ldots,p_d)\left( \R^\infty  \right)\]
for all $\omega \in P_d\left( \R^\infty \right)$.
Moreover, for all $N \in \N$ we can decompose $P_d(x)$ as
\[ P_d(x) = P_d(x^{(N)}) + P_d(x_{N+1},x_{N+2},\dots)\] and $\sum_{i=N+1}^\infty x_{i} \to 0$ for $N \to \infty$, since $x \in \ell^1$. 
So for $N \geq m$ sufficiently large $\Vert P_d(x_{N+1},x_{N+2},\dots) \Vert < \varepsilon $
which proves
\[ P_d(x) \in B_{\varepsilon}\left(P_d(x^{(m)})+P_d(x_{m+1},\dots,x_N)\right) \subset P_d(\R^\infty ).\]
\end{proof}

We point out that the images of $\ell^1$ and $\R^\infty$ under the map $P_\infty = (p_1,p_2,\ldots) : \ell^1 \to \R^{\N}$ are different and $P_\infty (\ell^1 \setminus \R^\infty) \cap P_\infty(\R^\infty) = \emptyset$ holds. An analytical-combinatorial proof can exploit the generating function identity 
\[E_{\X}(t):= 1 + \sum_{n = 1}^\infty e_n(\X) t^n = \prod_{i = 1}^\infty (1+X_it)\]
which can be found in \cite[Equation~(2.2)]{Macdonald}. 
For every $x \in \R^\infty$ the function $E_x(t)$ is a polynomial since all but finitely many factors in the infinite product are $1$. For every $z \in \ell^1 \setminus \R^\infty$ the infinite product $E_z(t)=\prod_{i=1}^\infty (1+z_it)$ is not a polynomial in $t$ and therefore distinct from all $E_x(t)$ with $x \in \R^\infty$. This implies that the images under the map $E_\infty = (e_1,e_2,\ldots) : \ell^1 \to \R^{\N}$ are different. From Newton's identities, one can deduce that the images under $P_\infty$ must also be different.

One might hope that the inclusion $(p_1,\ldots,p_d)(\R^n) \subset (p_1,\ldots,p_d)(\ell^1)$ is an equality for all sufficiently large $n$. The following example demonstrates that this is not possible.
\begin{example}
    Let $k \geq 1$ be an integer and $f_k = (p_1-k)^2+(p_2-k)^2 \in \Lambda.$ The infimum of $f_k$ on $\ell^1$ is $0$ and it is attained at $\mathbf{1}=(1,\ldots,1) \in \R^k$.
    We claim that $f_k^{(n)}$ is strictly positive on $\R^n$ for $n < k$. Note that $f_k^{(n)}(x) = 0$ implies $p_1^{(n)}(x)=p_2^{(n)}(x)=k$. By the Cauchy-Schwarz inequality 
    \[k^2 = p_1^{(n)}(x)^2 = \langle x, \mathbf{1} \rangle^2 \leq \Vert x\Vert^2 \cdot \Vert \mathbf{1} \Vert^2 = p_2(x) \cdot n = k \cdot n\]
    holds and we observe that $n \geq k$ must hold for $f_k^{(n)}$ having a root. 
\end{example}

\begin{remark}\label{rem: same images}
Lemma \ref{lem: images on ell ball versus unions of Rn} and its proof can be adapted to equality of the sets $(p_{k_1},\dots,p_{k_d})(\ell^{k_1}) $ and $(p_{k_1},\dots,p_{k_d})(\R^\infty) $ for positive integers $1 \leq k_1 < k_2 < \dots < k_d$. The proof becomes slightly more technical, and we use that the sequence $x \in \ell^{k_1}\setminus \R^\infty$ contains $d$ distinct non-zero coordinates and either all these $d$ scalars must be positive or negative. Instead of the Vandermonde determinant one encounters a generalized Vandermonde determinant as the Jacobian of the adaptation of $\nu_d$. Then the result follows from the definition of a Schur polynomial as the quotient of the generalized Vandermonde determinant by the Vandermonde determinant \cite[Chapter~I.3]{Macdonald} modulo some non-zero integer scalar. In particular, a symmetric function $f=g(p_{k_1},\dots,p_{k_d})$ is non-negative (positive) on $\ell^{k_1}$ if and only if $f$ is non-negative (positive) on $\ell^1$.    
\end{remark}

We immediately obtain the following corollary.

\begin{corollary}
   A symmetric function $f \in \Lambda$ is positive (non-negative) on $\ell^1$ if and only if $f$ is any-dimensional positive definite (any-dimensional psd).
\end{corollary}
\begin{proof}
Any $f \in \Lambda$ can be uniquely expressed as $f=g(p_1,\ldots,p_{2d})$ for some $d \in \N$ and some $2d$-variate real polynomial $g(Z_1,\ldots,Z_{2d})$. 
Thus, $f$ is positive (non-negative) if and only if $g$ is positive (non-negative) on the set $(p_1,\ldots,p_{2d})(\ell^1)= (p_1,\ldots,p_{2d})(\R^\infty)$,
where equality was proven in Lemma \ref{lem: images on ell ball versus unions of Rn}. 
\end{proof}

We recall \emph{Timofte's half-degree principle} concerning non-negativity of symmetric polynomials.
\begin{theorem}\cite{riener2012degree,timofte2003positivity}
Let $f \in \R[X_1,\ldots,X_n]$ be symmetric of degree $d$. For every $x \in \R^n$ there exists a point $y \in \R^n$ with at most $\max\left\{ \lfloor \frac{d}{2} \rfloor , 2\right\}$ distinct coordinates and $f(x)=f(y)$.
\end{theorem}
We can deduce that any point in $(p_1,\ldots,p_d)(\ell^1)$ is attained by an evaluation at a point in $\R^{\infty}$ with at most $\max\{\lfloor \frac{d}{2} \rfloor,2\}$ distinct coordinates. In related work \cite[Theorem~2]{krasikov2026generalized} the authors recently formulated that the truncation of the symmetric function $f=\sum_{i=0}^{2d} a_i e_ie_{2d-i}$ is non-negative for all numbers of variables if and only if $f$ is non-negative at all points in $\R^\infty$ with at most $d$ distinct coordinates.

The following alternative conclusion of Lemma \ref{lem: images on ell ball versus unions of Rn} demonstrates a stabilization in any-dimensional optimization problems in the ring $\Lambda$ if the infimum is attained.

\begin{corollary}
    Let $f \in \Lambda$ with $\alpha=\inf_{x \in \ell^1}f(x)$ being attained at some $x \in \ell^1$. Then there exists $N \in \N$ with $\alpha = \inf_{x \in \R^N}f^{(N)}(x)$. In particular, the sequence of polynomial optimization problems $\left( \inf_{x \in \R^n} f^{(n)} \right)$ stabilizes.     
\end{corollary}

If the infimum of $f \in \Lambda$ is not attained on $\ell^1$, then the sequence of polynomial optimization problems $\left( \inf_{x \in \R^n} f^{(n)} \right)$ does not stabilize.

\begin{example}
Consider $f=p_4+(p_2-1)^2 \in \Lambda$ which is positive on $\ell^1$. Evaluating $f^{(n)}$ at the points $x^{(n)}=\left( \frac{1}{\sqrt{n}},\dots,\frac{1}{\sqrt{n}} \right)$ gives $f^{(n)}(x^{(n)}) = \frac{1}{n}$ which converges to $0$. However, the polynomials $f^{(n)}$ are strictly positive.
\end{example}

\section{Positivstellensätze for any-dimensional positive symmetric polynomials}

In this section, we extend the two Positivstellensätze of Pólya \cite{polya} and Reznick \cite{reznick1995uniform} to the case of any-dimensional positive symmetric functions, proving that one can choose a uniform denominator for all numbers of variables to verify strict positivity of a symmetric function. 

\subsection{Any-dimensional Pólya's Positivstellensatz}

The assumption in Pólya's theorem is usually formulated as positivity of a homogeneous polynomial on all non-zero points of a simplex. By substituting $X_i^2$ for each $X_i$, one equivalently arrives at the formulation below.

\begin{theorem}[\cite{polya}~P{\'o}lya~1928]
Let $f\in \R[X_1,\dots,X_n]$ be homogeneous, even and positive definite. Then there is $k\in \N_0$ such that
\[\left(\sum_{i=1}^n X_i^2\right)^kf\]
has non-negative coefficients and is therefore a sum of squares of monomials.
\end{theorem}

Powers and Reznick proved a bound such that for any integer $k$ larger than the bound the polynomial $p_2^k \cdot f$ has non-negative coefficients. Their bound depends on the degree of $f$, the minimum of $f$ on the sphere, and the maximum absolute value among all coefficients of $f$. The bound does not have to be strict.

\begin{theorem}[\cite{powers2001new}~Theorem~1]\label{thm: powers and reznick}
Let $f = \sum_{\alpha \in \N_0^{n}} c_\alpha \prod_{i=1}^nX_i^{\alpha_i} \in \R[X_1,\ldots,X_n]$ be homogeneous, positive definite, even of degree $2d$, and let $L:= \max \left\{ |c_\alpha| ~\middle|~ \alpha \in \N_0^n, |\alpha|=2d\right\}, \lambda := \min \left\{ f(x) ~\middle|~ x \in \mathbb{S}^{n-1} \right\}$. Then for any $k \in \N$ with \[ k > \frac{d(d-1)}{2} \frac{L}{\lambda}-d\]
the polynomial $(\sum_{i=1}^n X_i^2)^kf$ has positive coefficients.
\end{theorem}

Because the bound by Powers and Reznick does not depend on the number of variables, we can directly deduce that every even homogeneous symmetric function $f \in \Lambda$, which is strictly greater than some $\varepsilon > 0$ on $\{x \in \R^\infty ~\mid~ p_2(x)=1\}$, has a dimension-independent certificate of its non-negativity. More specifically, $p_2^k f = \sum_{\lambda} c_\lambda m_\lambda$ for some $k \in \N, c_\lambda \geq 0$ and $\lambda$ ranges over all partitions of $2k+\deg (f)$ with only even parts. In particular, this truncates to an any-dimensional certificate of the positivity of the sequence of symmetric polynomials $(f^{(n)})$.

\begin{corollary}\label{cor:infinite Polya}
Let $f(\X)$ be a homogeneous even symmetric function. If 
\[f(x) \geq \varepsilon \text{ for all } x\in \R^\infty \text{ with } p_2(x)=1\]
for some fixed $\varepsilon > 0$, then there is $k \in \N$ such that $p_2^k f$ is a linear combination of even monomial symmetric functions with only non-negative scalars. In particular, $p_2^kf \in \SSigma$.
\end{corollary}
\begin{proof}
Since $f \in \Lambda$ is even of degree $2d$, there exists $g \in \R[Z_1,\ldots,Z_d]$ with $f=g(p_2,p_4,\ldots,p_{2d})$. 
Observe that the parameters $(d,L)$ from Theorem \ref{thm: powers and reznick} are identical for all truncations $f^{(n)}$. 
Let $\lambda_n = \min\{ f^{(n)}(x) ~|~ x \in \mathbb{S}^{n-1}\}$ and $\lambda_\infty = \inf \{ f(x) ~|~ x \in \R^\infty, p_2(x)=1\}$. Then $ \varepsilon \leq \lambda_\infty \leq \lambda_n$ for all $n \in \N$. In particular, for some fixed $k > \frac{d(d-1)}{2} \frac{L}{\lambda_\infty} -d $ and $N=2(d+k)$ the homogeneous polynomial $(p_2^{(N)})^kf^{(N)}$ of degree $N$ is a linear combination of even monomials, i.e., 
    \[ (p_2^{(N)})^kf^{(N)}= \sum_{\substack{\lambda \vdash 2(d+k) \\ 2 \mid \lambda_i }}c_\lambda m_\lambda^{(N)},  \]
    for some non-negative coefficients $c_\lambda$.
    Moreover, the base change of the power sum basis to the monomial symmetric basis in $(p_2^{(n)})^kf^{(n)}=(p_2^{(n)})^kg(p_2^{(n)},\dots,p_{2d}^{(n)})$ stabilizes for $n \geq N$. Thus, we find that
    $p_2^kf = \sum_{\lambda \vdash 2(d+k), 2 \mid \lambda_i} c_\lambda m_\lambda$ holds. 
\end{proof}

In \cite[Theorem~6.6]{acevedo2025wonderful} it was shown that deciding non-negativity of multi-homogeneous even symmetric functions in several groups of pairwise disjoint variables is undecidable. It remains an open question whether deciding non-negativity of homogeneous (even) symmetric functions is decidable. In contrast, the problem of deciding whether for a given parameter $\varepsilon > 0$ a homogeneous even symmetric function satisfies $f \geq \varepsilon $ is decidable by Corollary \ref{cor:infinite Polya}. 

We also obtain a non-homogeneous variant of the Positivstellensatz.
\begin{corollary}
    Let $f\in\R[p_2,p_4,\dots]$ be an even symmetric function and $\varepsilon > 0$ such that
    \[f^h(x_0,x)\geq \varepsilon \text{ for all } (x_0,x)\in \R^\infty \text{ with } x_0^2+p_2(x)=1,\] where $f^h$ denotes the homogenization of $f$ with respect to $X_0$. Then there is $k \in \N$ such that $(1+p_2)^k f$ is an any-dimensional sos.
\end{corollary}
\begin{proof}
    Again, we use Theorem \ref{thm: powers and reznick} to conclude that there exists a positive integer $k$ such that $(X_0^2+p_2^{(n)})^kf^h(X_0,X_1,\dots,X_n,0,0,\dots)$ is sos for all $n \in \N$. Thus, also its dehomogenization with respect $X_0$
    \[(1+p_2^{(n)})^k f^{(n)}\]
    is sos for all $n$. So, the symmetric function $(1+p_2)^kf$ is an any-dimensional sos. 
\end{proof}

\subsection{Any-dimensional Reznick's Positivstellensatz}

The following Positivstellensatz is due to Reznick and guarantees that the uniform denominator $p_2^{(n)}$, taken to some power $k$, can be used in a rational sum of squares representation of a homogeneous positive definite polynomial. 
\begin{theorem}[\cite{reznick1995uniform} Reznick 1995]\label{thm:reznick}
Let $f\in \R[X_1,\dots,X_n]$ be homogeneous and positive definite. Then there is $k\in \N_0$ such that
\[\left(\sum_{i=1}^n X_i^2\right)^kf\]
is a sum of squares.
\end{theorem}

\begin{remark}\label{rem:boundreznick}
Reznick actually proves a stronger result \cite[Theorem~3.12]{reznick1995uniform}. He shows that if $K\subseteq \R$ is an ordered field, $f\in K[X_1,\dots,X_n]$ is homogeneous and positive definite of degree $2d$ and 
\[k\geq \frac{n\cdot 2d(2d-1)}{(4\log 2)\varepsilon(f)}-\frac{n+2d}{2}, \quad \text{where} \quad \varepsilon(f):=\frac{\inf \{ f(x) ~|~ x \in \mathbb{S}^{n-1} \}}{\sup \{ f(x) ~|~ x \in \mathbb{S}^{n-1} \}},\] 
then $\left(\sum_{i=1}^n X_i^2\right)^kf$ is a non-negative $K$-linear combination of $(2d+2k)$-th powers of linear forms in $\Q[X_1,\dots,X_n]$.
\end{remark}
Again, it is possible that $(p_2^{(n)})^k f$ is sos for some $k$ less than the bound.
Unlike the bound by Powers and Reznick for Pólya's Positivstellensatz, the bound in Remark \ref{rem:boundreznick} for Theorem \ref{thm:reznick} depends on the number of variables. We are therefore not aware of a similarly simple proof of the following theorem as for the symmetric function analog of Pólya's Positivstellensatz (Corollary~\ref{cor:infinite Polya}).

\begin{theorem}\label{thm: dimension-independent reznick}
    Let $d \in \N$ and $\varepsilon > 0$. Let $f\in \R[p_2,p_3,\dots,p_{2d}]$ (i.e., $f$ is a symmetric function where $p_1$ does not appear) be homogeneous of degree $2d$ and suppose that 
    \[
    f(x)\geq \varepsilon \text{ for all } x\in \R^\infty \text{ with } p_2(x)=1.
    \]
    Then there exists $k\in \N_0$ such that $p_2^kf$ is an any-dimensional sos.
\end{theorem}
Our proof of the theorem relies on the representation theorem for Archimedean $T$-modules from Jacobi \cite{jacobi2001representation} and Marshall \cite{marshall2002general}. 

Before stating the representation theorem we recall basic terminology from real algebra and refer the interested reader to \cite[Chapters~2~and~5]{marshall2008positive} or \cite[Chapter~3]{scheiderer2024course}. Let $A$ be a commutative ring containing $\Q$.
A \emph{preordering} of $A$ is a subset $T$ of $A$ with $T+T\subset T, T\cdot T \subset T$ and $a^2 \in T$ for all $a \in A$. Note that $\sum A^2$ is the unique smallest preordering of $A$. A preordering $T$ is called \emph{Archimedean} if for all $a \in A$ there exists a positive integer $n$ with $n+ a \in T$. Let $T$ be a preordering of $A$. A subset $M$ of $A$ is called a \emph{$T$-module} if $M+M \subset M, T\cdot M \subset M$ and $1 \in M$. In particular, $T$ is a $T$-module and every $T$-module contains $T$. A $T$-module $M$ is \emph{Archimedean} if for all $a \in A$ there exists a positive integer $n$ with $n+a \in M$. Note that if $T$ is Archimedean then every $T$-module is Archimedean. 
Finally, we denote the set of all unitary ring homomorphisms $A \to \R$ which evaluate non-negatively on a set $S \subset A$ by $\mathcal{K}_S$, i.e.,
\[ \mathcal{K}_S = \{ \varphi \in \hom (A,\R) ~|~ \varphi (S) \subset [0,\infty)\}.\]
The set $\mathcal{K}_S$ is called the \emph{character space} of $S$. We will apply the following result.

\begin{theorem}[\cite{marshall2008positive}~5.4.4~Representation theorem] \label{thm: representation theorem}
Suppose $M \subset A$ is an Archimedean $T$-module and $T$ is a preordering of $A$. Then, every element $a \in A$ with $\varphi(a)>0$ for all $\varphi \in \mathcal{K}_M$ satisfies $a \in M$.
\end{theorem}
The representation theorem is actually more general and allows $T$ to be a quasi-preordering of $A$.

\begin{remark}
In Marshall's book \cite[Chapter~5.5]{marshall2008positive} a proof of Reznick's Positivstellensatz, as formulated in Theorem \ref{thm:reznick}, is presented. 
We follow closely Marshall's proof strategy but point out that the Archimedean property of the considered preordering $T$ and the classification of the set of all non-negative ring homomorphisms on $T$ seem non-trivial. In fact, the characterization of the character space $\mathcal{K}_T$ can be seen as a description of the closure of the orbit space  of $\left\{x \in \ell^2 ~\middle| ~ \Vert x\Vert_2=1 \right\}$ with respect to the infinite symmetric group similarly to Procesi's and Schwarz's characterization for the symmetric group \cite{procesi1985inequalities}.
\end{remark}

The outline of the proof of Theorem \ref{thm: dimension-independent reznick} is as follows.
In the ring \[A:=\R[p_2,p_3,\ldots] \subset \Lambda\] we define the preordering \[T := (\SSigma \cap A) + I,\] where $I$ is the principal ideal generated by $1-p_2$ in $A$. We show that $T$ is Archimedean in Lemma \ref{lem: T is Archimedean}. We prove that the character space 
\[\mathcal{K}_T:= \{ \varphi \in \hom_\R(A,\R) ~|~ \varphi(f) \geq 0 ~ \forall f \in T\} \]
is the closure of the set $\left(p_2,p_3,\ldots\right)\left(\left\{x\in \R^\infty ~\middle|~ p_2(x)=1\right\}\right)$ with respect to the product topology in Corollary \ref{cor:K_T-description}. 
Then it follows from the representation theorem that $f \in T$ and we  manipulate this certificate to conclude that $p_2^kf$ is a sos for some $k \in \N$. 

In the remaining part of this section we always consider the ring $A = \R[p_2,p_3,\ldots]$, the ideal $I = (1-p_2)$ and the preordering $T = (\SSigma \cap A)+I$ in $A$, and the set $\mathcal{K}_T = \{ \varphi \in \operatorname{Hom}(A,\R) ~\mid~ \varphi (T) \subset \R_{\geq 0}\}$ denotes the set of all unitary ring homomorphisms $A \to \R$ which are non-negative on $T$. Because $A$ is generated by the algebraically independent elements $p_2,p_3,\ldots$ every $\varphi \in \mathcal{K}_T$ is uniquely defined by the sequence $(\varphi(p_{n}))_{n \geq 2}$.  

Our classification of $\mathcal{K}_T$ relies on the identification of $(\varphi(p_{n}))_{n \geq 2}$ with the moments of certain discrete probability measures on the compact interval $[-1,1]$. We continue with the proof of Theorem \ref{thm: dimension-independent reznick}, but the proofs of the key ingredients mentioned above will be presented in Section \ref{sec: orbit space}.

\begin{proof}[Proof of Theorem \ref{thm: dimension-independent reznick}]
It follows from Remark \ref{rem: same images} and Corollary \ref{cor:K_T-description} that $f(x) \geq \varepsilon $ for all $x \in \R^{\infty}$ with $p_2(x)=1$ implies that $\varphi(f)>0$ for all $\varphi \in \mathcal{K}_T$. So, we can conclude that $f \in T$ by the Archimedeanity of $T$ (Lemma \ref{lem: T is Archimedean}) and the Representation Theorem \ref{thm: representation theorem}, i.e., 
$f = \sum_{i=1}^\infty g_i^2 +r \cdot (1-p_2) $, where $r \in A$ and $g_i \in \R[[\X]]_{\leq \ell}$ are formal power series of degree at most $\ell$ for some $\ell \in \N$.
We substitute $\frac{X_i}{\sqrt{p_2}}$ for each $X_i$ and obtain in the field $\operatorname{Quot}(\R[[\X]])(\sqrt{p_2})$
\begin{align*}
 \frac{1}{\sqrt{p_2}^{2d}}f(\X) &= f\left(\frac{X_1}{\sqrt{p_2}},\frac{X_2}{\sqrt{p_2}},\ldots\right) \\
  &=\sum_{i=1}^\infty g_i^2\left(\frac{X_1}{\sqrt{p_2}},\frac{X_2}{\sqrt{p_2}},\ldots\right) + r\left(\frac{X_1}{\sqrt{p_2}},\frac{X_2}{\sqrt{p_2}},\ldots\right) \cdot \left( 1- \sum_{i=1}^\infty \frac{X_i^2}{\sqrt{p_2}^2}\right) \\  
 & = \sum_{i=1}^\infty g_i^2\left(\frac{X_1}{\sqrt{p_2}},\frac{X_2}{\sqrt{p_2}},\ldots\right)+r\left(\frac{X_1}{\sqrt{p_2}},\frac{X_2}{\sqrt{p_2}},\ldots\right) \cdot \left( 1-  \frac{p_2}{p_2}\right) \\
 & = \sum_{i=1}^\infty g_i^2\left(\frac{X_1}{\sqrt{p_2}},\frac{X_2}{\sqrt{p_2}},\ldots\right) .
\end{align*}
We multiply both sides of the equation with $p_2^{2N}$ for sufficiently large $N$ and obtain 
\begin{align}\label{eq: p2N times f}
p_2^{2N-d} f(\X) = \sum_{i=1}^\infty p_2^{2N}g_i^2\left(\frac{X_1}{\sqrt{p_2}},\frac{X_2}{\sqrt{p_2}},\ldots\right) = \sum_{i=1}^\infty \left( g_{i,1}(\X)+g_{i,2}(\X)\sqrt{p_2}\right)^2    
\end{align}
where $g_{i,1},g_{i,2} \in \R[[\X]]$ are of degree at most $2N$. This can be seen as follows, where $\X^\alpha := \prod_{i=1}^\infty X_i^{\alpha_i}$:
\[ p_2^{N} \sum_{\substack{\alpha \in \N_0^{\N} \\ |\alpha| \leq \ell}} c_\alpha \prod_{i=1}^\infty \frac{X_i^{\alpha_i}}{\sqrt{p_2}^{\alpha_i}}
=\sum_{\substack{\alpha \in \N_0^{\N} \\ |\alpha| \leq \ell}} c_\alpha p_2^{N} \frac{1}{\sqrt{p_2}^{|\alpha|}} \X^\alpha = \sum_{\substack{\alpha \in \N_0^{\N}, |\alpha| \leq \ell \\ |\alpha| \in 2\N}}c_\alpha p_2^{N-\frac{|\alpha|}{2}} \X^\alpha+\sum_{\substack{\alpha \in \N_0^{\N}, |\alpha| \leq \ell \\ |\alpha| \not \in 2\N}}c_\alpha p_2^{N-\left\lceil\frac{|\alpha|}{2}\right\rceil} \X^\alpha \sqrt{p_2}.\]
Thus, equation (\ref{eq: p2N times f}) implies
\[ p_2^{2N-d}f(\X) = \sum_{i=1}^\infty (g_{i,1}(\X)^2+2g_{i,1}(\X)g_{i,2}(\X)\sqrt{p_2}+g_{i,2}(\X)^2p_2).\]

Finally, we verify that $p_2$ is not a square in $\operatorname{Quot}\left(\R\left[\left[\X\right]\right]\right)$ which shows that $\sum_{i=1}^\infty g_{i,1}(\X)g_{i,2}(\X)=0$ must hold in equation (\ref{eq: p2N times f}). This will prove that
$p_2^{2N-d}f(\X) \in \SSigma.$ 
Suppose that there are $r,s \in \R[[\X]]\setminus\{0\}$ with $r^2 = p_2s^2$. The element $p_2$ is irreducible in $\R[[\X]]$. Indeed, comparing lowest degree homogeneous parts, a nontrivial factorization would force $p_2$ to be a product of two real linear forms, which is impossible. Furthermore, $\R[[\X]]$ is a UFD by a theorem of Nishimura \cite{nishimura1967unique}, therefore the element $p_2$ is prime.\footnote{Alternatively, one can truncate to finitely many variables, where both the UFD property and the irreducibility of $p_2^{(n)}$ are standard facts.} Now the exponent of $p_2$ in the factorization of $r^2$ is even, whereas the exponent of $p_2$ in the factorization of $p_2s^2$ is odd, which is a contradiction.
\end{proof}

Next, we verify that $T$ is indeed Archimedean.
\begin{lemma}\label{lem: T is Archimedean}
    The set $T = (\SSigma \cap A) + I \subset A$ is an Archimedean preordering.
\end{lemma}
\begin{proof}
We first show that $T$ is a \emph{preordering}. Therefore, we need to verify that $T + T$, $T\cdot T$ and $A^2$ are subsets of $T$. The first and last inclusion is clear, while we also have 
\[(\sigma_1+r_1)(\sigma_2+r_2) =\sigma_1\sigma_2 +r_1(\sigma_2+r_2)+r_2\sigma_1 \in T \]
for all $\sigma_1,\sigma_2 \in \SSigma \cap A$ and $r_1,r_2 \in I$. \\
Next, we prove that $T$ is Archimedean.
    We show that $1 \pm p_k \in T$ for all $k \geq 2$.
    This suffices because the set 
    \[H_T := \{a \in A ~\mid~ \exists n \in \N~ n \pm a \in T\}\] contains $\R_{\geq 0}$ and is a subring of $A$ by \cite[5.2.3~Proposition]{marshall2008positive}. \\
    Observe that $1 \pm p_k = 1-p_2 + (p_2 \pm p_k)$. Because $1 - p_2 \in T$, we can alternatively verify that $p_2 \pm p_k \in T$ holds. We prove this by extending a representation of an univariate polynomial $t^2 \pm t^k\in \R[t]$ as a sum of squares plus an element in the ideal $(1-t^2)\subset \R[t]$ to a certificate for $p_2 \pm p_k \in T$. 
    Note that $t^2(1-t^{k-2}) \geq 0$ on the set $[-1,1]$. By the Markov-Lukács theorem \cite{lukacs1918verscharfung} and  \cite[Theorem~1.21.1]{szego1975orthogonal} we obtain 
    \[ t^2(1 \pm t^{k-2})  = s_{\pm}(t)+r_{\pm}(t)(1-t^2)\]
    for some sums of squares $s_{\pm}(t), r_{\pm}(t) \in \R[t]$.\footnote{In the literature, the Markov-Lukács theorem is often formulated differently. For $f(t) \geq 0$ on $[-1,1]$ there are $p,q,s,r \in \R[t]$ such that $f=p^2+(1-t^2)q^2$ (if $\deg f$ is even) and $f=(1-t)s^2+(1+t)r^2$ (if $\deg f$ is odd). As observed in \cite[Corollary~3]{powers2000polynomials} the identity $1 \pm t = \frac{(1 \pm t)^2}{2} + \frac{1}{2}(1-t^2)$ proves that we can independently of the parity of $\deg(f)$ find a representation of $f$ of the form $\sigma_1+(1-t^2)\sigma_2$, where $\sigma_1,\sigma_2 \in \R[t]$ are sos.}
    We observe that $\sum_{i=1}^\infty s_{\pm}(X_i)$ and $\sum_{i=1}^\infty r_{\pm}(X_i) \in A$ holds. This is because $0=s_{\pm}(0)+r_{\pm}(0) $ and $s_{\pm}(0)\geq 0, r_{\pm}(0) \geq 0$ implies that $s_{\pm}(t)$ and $r_{\pm}(t)$ have no non-zero constant term. Since these are sums of squares, they also cannot have a linear term. Thus,
    \begin{align*}
    p_2 \pm p_k& = \sum_{i=1}^\infty X_i^2(1 \pm X_i^{k-2}) = \sum_{i=1}^\infty s_{\pm}(X_i) + \sum_{i=1}^\infty r_{\pm}(X_i)\left(1-X_i^2\right) \\
    & = \sum_{i=1}^\infty s_{\pm} (X_i) + \sum_{i=1}^\infty r_{\pm}(X_i)\Big(1-p_2+\sum_{j \neq i} X_j^2\Big) = \sum_{i=1}^\infty \left( s_{\pm}(X_i)+r_{\pm}(X_i)\sum_{j\neq i}X_j^2\right) + (1-p_2)\sum_{i=1}^\infty r_{\pm}(X_i)   \\
    & = \sigma_1+ \sigma_2 (1-p_2),
    \end{align*} 
    where $\sigma_1,\sigma_2 \in \SSigma$. This proves $p_2 \pm p_k \in T$. 
\end{proof}

Since $p_1$ is not bounded on the set $\{x \in \R^\infty ~|~ p_2(x)=1\}$, we cannot directly adapt the proof of Theorem \ref{thm: dimension-independent reznick} to the full ring $\Lambda$. The respective preordering $T=\SSigma + (1-p_2)$ would not be Archimedean. 

In Corollary \ref{cor:K_T-description with p_1}, we will show that the value of $p_1$ is free over $\{x \in \R^\infty ~|~ p_2(x)=1\}$, which allows us to use a cylinder positivstellensatz to prove in Theorem \ref{thm: dimension-independent reznick with p_1} a generalization including $p_1$.

\section{\texorpdfstring{The $S_\infty$-orbit space, $\mathcal{K}_T$}{KT} and discrete probability measures on \texorpdfstring{$[-1,1]$}{[-1,1]}}\label{sec: orbit space}
Throughout this section we consider the ring $A = \R[p_2,p_3,\ldots]$, the ideal $I = (1-p_2) \subset A$, the preordering $T = (\SSigma \cap A)+I$ in $A$, and the character space $\mathcal{K}_T = \left\{ \varphi \in \operatorname{Hom}(A,\R) ~\middle| ~ \varphi (T) \subset \R_{\geq 0}\right\}$ of all ring homomorphisms $A \to \R$ which are non-negative on $T$. 
We investigate the relation between discrete probability measures on the interval $[-1,1]$ and the evaluations of the map 
\[(p_2,p_3,\ldots) : \ell^2 \to \R^{\N}.\] 
This leads us to a characterization of $\mathcal{K}_T$.
Moreover, let $\S_\infty = \{ \vartheta : \N \to \N ~\mid~ \vartheta \text{ bijective}\}$ denote the \emph{large infinite symmetric group} of all permutations of $\N$. We can view the set $(p_2,p_3,\ldots)(\ell^2)$ as the orbit space of $\ell^2$ with respect to $\S_\infty$ as the following proposition demonstrates. 

\begin{proposition}\label{prop: orbit space}
    Let $\widetilde{x}$, $\widetilde{y} \in \ell^2$ with $(p_2,p_3,\ldots)(\widetilde{x})=(p_2,p_3,\ldots)(\widetilde{y})$. We denote the truncations of $\widetilde{x}$ and $\widetilde{y}$ where all $0$-coordinates are deleted by $x$ and $y$. Then $x$ and $y$ are in the same orbit of $\S_\infty$, after embedding $x \mapsto (x,0,0,\ldots)$ and $y \mapsto (y,0,0,\ldots)$.  
\end{proposition}

\begin{proof}
    We conclude from the generating function identity
    \[ P_{\X}(t) = \sum_{n = 1}^\infty p_n t^{n-1} = \sum_{i=1}^\infty \frac{X_i}{1-X_it},\]
    which can be found in \cite[Page~23]{Macdonald}, that
    \[ Q_{\X}(t) = \sum_{n = 2}^\infty p_n t^{n-1} = \sum_{i=1}^\infty \frac{X_i}{1-X_it} - \sum_{i=1}^\infty X_i = \sum_{i =1}^\infty \frac{X_i^2t}{1-X_it}. \]
    We claim that for all $z \in \ell^2$ the function $Q_z(t)$ is meromorphic on $\C$. If $z \in \R^n$ for some $n \in \N$ the function $Q_z(t)$ is already a rational function and thus meromorphic on $\C$. So, suppose $z \in \ell^2$ with no $0$ coordinate. 
    Since $\sum_{i=1}^\infty z_i^2 < \infty$ we can conclude that $z_i \to 0$ for $i \to \infty$. In particular, we have $\frac{1}{|z_i|} \to \infty$ for $i \to \infty$. So, the sequence $\left(\frac{1}{z_i}\right)_{i \in \N}$ has no finite accumulation point and all potential poles $\frac{1}{z_i}$ of $Q_z(t)$ are isolated. Next, we prove uniform convergence on any compact set $K \subset \C$ which does not contain a point from the sequence $\left(\frac{1}{z_i}\right)_{i}$. Let $R > 0$ such that $|t| \leq R$ for all $t \in K$. There exists $N \in \N$ such that $|z_i| \leq \frac{1}{2R}$ for all $i \geq N$. The inequalities  
    \[ |1-z_it| \geq \left|1-\left|z_i\right|R\right| \geq \frac{1}{2} \quad \text{and} \quad \left| \frac{z_i^2t}{1-z_it} \right| \leq \frac{Rz_i^2}{\frac{1}{2}}=2Rz_i^2\]
    hold on $K$ for all $i \geq N$ and the sum $\sum_{i=1}^\infty 2Rz_i^2$ converges. It follows from the Weierstrass M-test that $\sum_{i=1}^n \frac{z_i^2t}{1-z_it}$ converges uniformly on $K$. So, $Q_z(t)$ is analytic on $\C \setminus \left\{ \frac{1}{z_i} ~\middle|~ i \in \N\right\}$ and therefore meromorphic.

    So, if $Q_x(t)=Q_y(t)$ then the functions have the same poles with the same residues. Since all poles are isolated, their residue at a pole $c$ is precisely the negative of the multiplicity of $\frac{1}{c}$ in $x$, respectively in $y$. This proves that indeed $x$ and $y$ are in the same $\S_\infty$-orbit.
\end{proof}

\begin{definition}
    We identify points $x=(x_1,x_2,\ldots) \in \ell^2$ with $\Vert x\Vert_2 =1$ with discrete probability measures $\mu_x := \sum_{i=1}^\infty x_i^2 \cdot \delta_{x_i}$ on the compact interval $[-1,1]$.
\end{definition}
Note that $\mu_x$ is indeed a discrete probability measure on $[-1,1]$, because $\Vert x\Vert_2=1$ implies $|x_i| \leq 1$ for all $i \in \N$ and its total mass is $\sum_{i=1}^\infty x_i^2 =p_2(x) =1$.

The following lemma identifies certain properties of measures $\mu_x$ for $x \in \ell^2$.
\begin{lemma}\label{lem: point evaluations and discrete measures}
    Let $x\in \ell^2$ with $\Vert x\Vert_2=1$ and let $\mu_x$ denote its associated discrete measure.
    Then the following assertions hold.
    \begin{enumerate}
        \item 
        The $j$-th moment $m_j:= \int t^j\mathrm{d}\mu_x(t)$ of $\mu_x$ equals $p_{j+2}(x)$.

        \item Every non-zero coordinate $x_i$ of $x$ is an isolated point in the sequence $x$.
        \item For a non-zero coordinate $x_i$ we have $\frac{\mu_x(x_i)}{x_i^2}\in \N$. In particular, the ratio $\frac{\mu_x(x_i)}{x_i^2}$ counts the multiplicity of $x_i$ in the sequence $x$.
    \end{enumerate}
\end{lemma}
\begin{proof}

\begin{enumerate}
    \item
    An elementary calculation shows 
    \[    m_j= \int t^j\mathrm{d}\mu_x(t)=\sum_{i=1}^\infty x_i^2x_i^j=p_{j+2}(x).\]
    \item Suppose that there exist $N \in \N$ with $x_N \neq 0$ and $\varepsilon > 0$ such that $r:=\min\{ |x_N \pm \varepsilon| \}> 0$ and $ \big\{ i \in \N ~\big|~ | x_N-x_i| < \varepsilon \big\}$ has infinite cardinality.
    Then \[1 = \sum_{i=1}^\infty x_i^2 \geq \sum_{\substack{i\in \N \\ |x_N-x_i| <\varepsilon}} x_i^2 \geq \sum_{\substack{i\in \N \\ |x_N-x_i| <\varepsilon}} r^2 = \infty \]
    which is a contradiction.
    \item It follows from (2) that each $x_i \neq 0$ is isolated. Thus, we find that $\frac{\mu_x(x_i)}{x_i^2}=\sum_{\substack{n \in \N \\ x_n = x_i}}1 $ is a natural number.
\end{enumerate}
\end{proof}

We need the following auxiliary result. 
\begin{lemma}\label{lem: e_2(..) in A}
Let $h \in \R[t]$ be a polynomial. Then the symmetric function \[e_2(X_1^2h(X_1)^2,X_2^2h(X_2)^2,\dots)\] is contained in $\SSigma \cap A$.   
\end{lemma}
\begin{proof}
Let $h(t)^2 = \sum_{i=0}^{2d} a_it^i$ and $h(t)^4 = \sum_{i=0}^{4d} b_it^i$. We verify that $2e_2(X_1^2h(X_1)^2,X_2^2h(X_2)^2,\dots) \in \SSigma \cap A$.
Note that
\[ 2e_2(X_1^2h(X_1)^2,X_2^2h(X_2)^2,\dots) = \sum_{i < j} 2(X_iX_jh(X_i)h(X_j))^2 \in \SSigma.\]
We use the base change $2e_2 = p_1^2-p_2$ to verify the containment in $A$. 
\begin{align*}
     2e_2(X_1^2h(X_1)^2,X_2^2h(X_2)^2,\dots) & = p_1(X_1^2h(X_1)^2,X_2^2h(X_2)^2,\dots)^2-p_2(X_1^2h(X_1)^2,X_2^2h(X_2)^2,\dots) \\
     & = \left( \sum_{i=1}^\infty X_i^2h(X_i)^2 \right)^2 -  \sum_{i=1}^\infty X_i^4h(X_i)^4 = \left( \sum_{i=0}^{2d} a_i p_{2+i} \right)^2- \sum_{i=0}^{4d} b_i p_{4+i} \in A.
\end{align*}
\end{proof}
The following auxiliary lemma can be found in \cite[Chapter~7]{Stanley} and \cite[Chapter~I.2]{Macdonald}. 
In the language of representation theory it says that the specialization of the basis transformation map $p_k \mapsto e_k$ at $p_i=t$ for all $i$ evaluates to the falling factorial $(t)_k$ divided by $k!$. A proof usually exploits the exponential identity linking the elementary symmetric polynomials to the power sums.
For completeness, we present a proof using Newton's identities. 

\begin{lemma}\label{lem: Luca's formula}
Let $k \geq 1$ be a positive integer and write $e_k = g(p_1,\ldots,p_k)$ for some $k$-variate polynomial $g$. Then $g(t,\ldots,t)=\frac{1}{k!}\prod_{i=0}^{k-1}(t-i)$.    
\end{lemma}
\begin{proof}
Newton's identities characterize the polynomial $g$:
\[ e_k = \frac{1}{k!} \det \begin{pmatrix} p_1 & 1 & 0 & 0 & \cdots & 0 \\ p_2 & p_1 & 2 & 0 & \cdots & 0 \\ p_3 & p_2 & p_1 & 3 & \ddots & \vdots \\ \vdots & \vdots & \vdots & \ddots & \ddots & 0 \\ p_{k-1}& p_{k-2}& p_{k-3}& \cdots & p_1 & k-1 \\ p_k & p_{k-1}& p_{k-2}& \cdots & p_2 & p_1 \end{pmatrix}.\]
Let $A \in \R[t]^{k \times k}$ denote the matrix arising from replacing each $p_i$ by $t$. We apply elementary column operations and replace each column $C_i$ by $C_i-C_{i-1}$ for $2 \leq i \leq k$. The resulting upper-triangular matrix
\[ B = \begin{pmatrix} t & 1 & 0 & 0 & \cdots & 0 \\ 0 & t-1 & 2 & 0 & \cdots & 0 \\ 0 & 0 & t-2 & 3 & \ddots & \vdots \\ \vdots & \vdots & \ddots & \ddots & \ddots & 0 \\ 0 & 0 & \cdots & 0 & t-(k-2) & k-1 \\ 0 & 0 & \cdots & 0 & 0 & t-(k-1) \end{pmatrix} \]
has the same determinant as $A$. 
The claim follows now directly.
\end{proof}

The next lemma shows that every $\varphi \in \mathcal{K}_T$ is the limit of measures $\mu_x$ with $x \in \ell^2$ and $\Vert x \Vert_2=1$.
\begin{lemma} \label{lem: K_T classification}
Let $T = (\SSigma \cap A) + I$, $\varphi\in \mathcal{K}_T$, and define 
$(m_j)_{j \in \N_0}:=\left(\varphi(p_{j+2})\right)_{j\in\N_0}.$ Then the following assertions hold.
\begin{enumerate}
    \item The sequence $(m_j)$ is the moment sequence of a probability measure on $[-1,1]$.
    \item If $(m_j)$ is the moment sequence of a measure $\mu$ on $[-1,1]$ and $h\in\R[t]$ is a polynomial with no linear or constant term, then
    \[\varphi\left(\sum_{i=1}^\infty h(X_i)\right)= \int \frac{h(t)}{t^2} \mathrm{d}\mu(t).\]
    \item If $(m_j)$ is the moment sequence of a probability measure $\mu$ on $[-1,1]$, then $\mu$ is a discrete measure.
    \item
    If $(m_j)$ is the moment sequence of a discrete probability measure $\mu$ on $[-1,1]$, then every $x\in [-1,1]$ with $x\neq 0$ and $\mu(x)> 0$ is an isolated point in the support of $\mu$.
    \item If $(m_j)$ is the moment sequence of a discrete probability measure $\mu$ on $[-1,1]$, then for every $x\in [-1,1]$ with $x\neq 0$ and $\mu(x)\neq 0$ we have $\frac{\mu({x})}{x^2}\in \N$. 
    \item If $\mu = c \cdot \delta_0 + \sum_{i \geq 1} N_i x_i^2 \cdot \delta_{x_i}$ is a discrete probability measure on $[-1,1]$ with $0 \neq x_i \neq x_j$ for all $i \neq j$, then its moment sequence $(m_j)_{j \in \N_0}$ lies in the closure (with respect to the product topology on $\R^\N$) of the set \[\left\{ (p_{2+j}(y))_{j \in \N_0} ~ \middle| ~ y \in \R^\infty, p_2(y)=1 \right\}.\] 
\end{enumerate}
\end{lemma}
\begin{proof}
    \begin{enumerate}
        \item We assume that $\left(m_{j}\right)_{j\in\N_0}$
    is not the moment sequence of a probability measure on $[-1,1]$ and claim that there exists $f\in T$ such that $\varphi(f)<0$. We define the linear function
    \[L_m: \R[t]\to \R, t^j\mapsto m_j.\]
    By the Riesz-Haviland theorem \cite{haviland1936momentum,riesz1923probleme}, there exists $g\in \R[t]$ such that $g([-1,1])\subseteq \R_{\geq 0}$ and $L_m(g)<0$. It follows from the Markov-Lukács theorem that we can write
    \[g=\sigma_1 + (1-t^2)\sigma_2\]
    for some sos $\sigma_1,\sigma_2 \in \R[t]$.\\
    \underline{Case 1:} There is $h\in \R[t]$ such that $L_m(h^2)<0$. Write $h(t)^2 = \sum_{k=0}^{2d}a_kt^k$ and define \[f:=\sum_{i=1}^\infty X_i^2h(X_i)^2 =\sum_{k=0}^{2d}a_kp_{k+2} \in A.\] Then
    \[\varphi(f)=\sum_{k=0}^{2d} a_k m_k = L_m(h^2)<0\]
    and by construction $f \in A \cap \SSigma \subset T$.
    \\
    \underline{Case 2:} There is $h\in \R[t]$ such that $L_m((1-t^2)h^2)<0$. 
    Write $h(t)^2 = \sum_{k=0}^{2d}a_kt^k$ and define
    \[f := \sum_{i=1}^\infty X_i^2(1-X_i^2)h(X_i)^2 = \sum_{k=0}^{2d}a_k(p_{k+2}-p_{k+4}) \in A. \]
    Then \[
    \varphi (f) = \sum_{k=0}^{2d}a_k(m_k-m_{k+2})=L_m((1-t^2)h^2) < 0.
    \]
   Moreover, the following identities
\begin{align*}
    f &= \sum_{i=1}^\infty X_i^2h(X_i)^2 - \sum_{i=1}^\infty X_i^4 h(X_i)^2 = (1-p_2) \cdot\sum_{i=1}^\infty X_i^2h(X_i)^2 + p_2 \cdot \sum_{i=1}^\infty X_i^2h(X_i)^2-\sum_{i=1}^\infty X_i^4 h(X_i)^2 \\
    & = (1-p_2) \cdot\sum_{i=1}^\infty X_i^2h(X_i)^2 +\sum_{i,j=1}^\infty X_j^2 X_i^2h(X_i)^2-\sum_{i=1}^\infty X_i^4 h(X_i)^2 \\
    &= (1-p_2) \cdot\sum_{i=1}^\infty X_i^2h(X_i)^2+\sum_{i < j }^\infty X_i^2X_j^2(h(X_i)^2+h(X_j)^2)
\end{align*}
demonstrate that $f \in I + \SSigma = T$ holds.

\item We write $h(t)=\sum_{j=0}^d a_jt^{j+2}$ and compute
    \begin{align*}
    \varphi\left(\sum_{i=1}^\infty h(X_i)\right)&=\varphi\left(\sum_{i=1}^\infty \sum_{j=0}^d a_j X_i^{j+2}\right)=\varphi\left(\sum_{j=0}^d a_j \sum_{i=1}^\infty X_i^{j+2}\right)=\sum_{j=0}^d a_j \varphi(p_{j+2}) \\
    &=\sum_{j=0}^d a_j m_j=\sum_{j=0}^d a_j \int t^j \mathrm{d} \mu(t) = \int \sum_{j=0}^d a_j t^j \mathrm{d} \mu(t)=\int \frac{h(t)}{t^2} \mathrm{d}\mu(t).
    \end{align*}
    \item 
    Suppose that $\mu$ is not a discrete measure on $[-1,1]$. We will show that this contradicts $\varphi \in \mathcal{K}_T$.
    Because $\mu$ is not a discrete measure, its continuous component is not identically zero. Moreover, any finite measure contains at most countably many atoms, i.e., points with positive mass. Therefore, there exists a non-atomic point $x_0 \in (-1,1)\setminus\{0\}$ that lies in the support of the continuous part of $\mu$. 

    For any $\delta > 0$ we consider the open neighborhood $I_\delta := (x_0-\delta,x_0+\delta) \cap [-1,1]$ of $x_0$. Since $x_0$ is in the support of the continuous component of $\mu$ we have $\mu(I_\delta) > 0$. Because $\mu(x_0)=0$ the continuity of the measure from above implies $\mu(I_\delta) \to 0$ for $\delta \downarrow 0$.
    Let $\kappa_\delta := \inf_{x \in I_\delta}x^2$. Since $x_0 \neq 0$ we have $\kappa_\delta \to x_0^2>0$ for $\delta \downarrow 0$.
    Because the set of atoms of $\mu$ is at most countable, we can restrict in our analysis to parameters $\delta$ such that $x_0-\delta$ and $x_0+\delta$ are not atoms of $\mu$. 
    By construction, we obtain the inequality 
    \begin{align}
    0 < w_\delta := \mu(I_\delta) < \kappa_\delta \label{eq:inequality w kappa}   
    \end{align}
    for sufficiently small $\delta$. This implies $w_\delta^2-w_\delta\kappa_\delta < 0$ for sufficiently small values of $\delta$. Then the quadratic polynomial $E_\delta(\varepsilon) = \varepsilon^2+2w_\delta \varepsilon+(w_\delta^2-\kappa_\delta w_\delta)$ satisfies $E_\delta(0)< 0$.
    By continuity there exists $\varepsilon_\delta > 0$ (which depends on $\delta$) such that $E_\delta(\varepsilon_\delta) < 0$.
    
    We now fix a $\delta > 0$ for which (\ref{eq:inequality w kappa}) holds and for which $0 \not \in \operatorname{cl} I_\delta$. We also fix $\varepsilon:=\varepsilon_\delta> 0$ such that $E_\delta(\varepsilon)<0$. Let $a < b$ be the endpoints of the closure of the interval $I_\delta$, i.e., $I_\delta=(a,b)$. By assumption $\mu(a)=\mu(b)=0$ which implies \[w:=w_\delta=\mu((a,b))=\mu([a,b]) < \kappa_\delta=:\kappa.\] 
    Continuity of the measure from above implies that there exists $\eta > 0$ such that $(a-\eta,b+\eta)\subset [-1,1]$ and $\mu((a-\eta,b+\eta))< w+\varepsilon$. 

    We consider the continuous, piecewise-linear function
    \[ f : [-1,1] \to \R, x \mapsto  \begin{cases} 
1 & \text{if } x \in (a,b) \\ 
1 - \frac{a-x}{\eta} & \text{if } x \in (a-\eta, a] \\
1 - \frac{x-b}{\eta} & \text{if } x \in [b, b+\eta) \\
0 & \text{otherwise.}
\end{cases} \]
By construction, $0 \leq f(x) \leq 1$ for all $x \in [-1,1]$. Moreover, we obtain 
\begin{align*}
    &\int f(t)^2 \mathrm{d} \mu(t) = \mu((a,b)) + \int_{(a-\eta,b+\eta) \setminus (a,b)} f(t)^2  \mathrm{d} \mu(t) = w + R_1, \\
    &\int t^2 f(t)^4 \mathrm{d} \mu(t) = \int_{(a,b)} t^2 \cdot 1^4 \mathrm{d}\mu(t) + \int_{(a-\eta,b+\eta)\setminus (a,b)} t^2 f(t)^4 \mathrm{d} \mu(t)   \geq \kappa \cdot w.
\end{align*}
We can bound $R_1$ as follows
\[ 0 \leq R_1 \leq \mu((a-\eta,b+\eta)\setminus (a,b)) < \varepsilon.\]
This implies
\begin{align}\label{eq:discrete}
\left( \int f(t)^2 \mathrm{d} \mu (t) \right)^2 - \int t^2 f(t)^4 \mathrm{d} \mu (t) \leq (w+\varepsilon)^2-\kappa \cdot w = w^2+w(2\varepsilon-\kappa)+\varepsilon^2 =E_\delta(\varepsilon)<0.    
\end{align}
It follows from the Weierstrass approximation theorem that $f$ can be uniformly approximated on the compact interval $[-1, 1]$ by a polynomial. Thus, there exists a polynomial $h \in \R[t]$ for which the inequality (\ref{eq:discrete}) is preserved, i.e.,
\[ \left( \int h(t)^2 \mathrm{d} \mu (t) \right)^2 - \int t^2 h(t)^4 \mathrm{d} \mu (t) < 0.  \]
Now we evaluate $\varphi$ on the symmetric function $e_2(X_1^2h(X_1)^2,X_2^2h(X_2)^2,\ldots)$ which is contained in $T$ by Lemma \ref{lem: e_2(..) in A}. We obtain
    \begin{align*}
    2\varphi(e_2(X_1^2 h(X_1)^2,X_2^2 h(X_2)^2,\ldots))
    &=\varphi\left( \left(\sum X_i^2 h(X_i)^2\right)^2-\sum (X_i^2 h(X_i)^2)^2\right) \\
    &=\varphi \left(\sum X_i^2 h(X_i)^2\right)^2-\varphi\left(\sum (X_i^2 h(X_i)^2)^2\right)     \\
    &=\left(\int h(t)^2 \mathrm{d} \mu(t) \right)^2- \int t^2h(t)^4 \mathrm{d} \mu(t)  < 0 
    \end{align*}
    which is a contradiction to $\varphi \in \mathcal{K}_T$.
    \item
    Write $\mu=\sum_{i=1}^\infty a_i \delta_{x_i}$ with $x_i \in [-1,1]$ and $a_i \geq 0$ with $\sum_{i=1}^\infty a_i =1$. 
    If only finitely many weights are positive, we obtain $\mu = \sum_{i=1}^n a_i \delta_{x_i}$ for some finite $n$ and the statement follows directly. So, we can suppose that all weights are positive and all $x_i$ are pairwise distinct. 
    We suppose that the support of $\mu$ has an accumulation point that is not equal to $0$. Without loss of generality we assume that $x_1\neq 0$ is not isolated. 
    So there is a subsequence $(x_{n_k})_k$ of $(x_n)_n$ which converges to $x_1$. Since the sum of the weights is $1$, the subsequence of weights $(a_{n_k})_k$ must converge to zero.
    Because $x_{n_k} \to x_1$ and $a_{n_k} \to 0$ for $k \to \infty$, we find that $\frac{a_{n_k}}{x_{n_k}^2} \to 0$ for $k \to \infty$. In particular, there exists a sufficiently large $K$ with $0<\frac{a_{n_K}}{x_{n_K}^2}<1$. We write $y:=x_{n_K}, a:=a_{n_K}$ and $w:=\frac{a}{y^2} \in (0,1)$. Then $w^2-w < 0$ holds. 
Consider the quadratic polynomial $E(\varepsilon)=w^2+w\left(2\frac{\varepsilon}{y^2}-1\right)+\frac{\varepsilon^2}{y^4}$ with $E(0)=w^2-w < 0$. 
So there exists a sufficiently small $\varepsilon > 0$ with $E(\varepsilon) < 0$.
Because $\mu$ is a finite measure there exists $\rho > 0$ such that $0 \not \in (y-\rho,y+\rho)\subset[-1,1]$ and $\mu \left( (y-\rho,y+\rho)\setminus \{y\} \right) < \varepsilon$.   
    We consider the continuous function 
    \[ f : [-1,1] \to \R, x \mapsto \begin{cases} \frac{1}{|y|} \left( 1 -\frac{|x-y|}{\rho} \right) & \text{ if } |x-y| < \rho \\ 0 & \text{ if } |x-y| \geq \rho 
\end{cases}\]
which satisfies $f(y)=\frac{1}{|y|}, 0 \leq f  \leq \frac{1}{|y|}$ and $f(x)=0$ for all $x \not \in (y-\rho,y+\rho)$.
In particular, $f$ is a piecewise-linear function. We calculate
\begin{align*}
    &\int f(t)^2 \mathrm{d} \mu (t) = \frac{a}{y^2} + \int_{(y-\rho,y+\rho)\setminus \{y\}}f(t)^2 \mathrm{d} \mu(t) = w + R_1, \\
    & \int t^2 f(t)^4 \mathrm{d} \mu (t) = \frac{a}{y^2}+ \int_{(y-\rho,y+\rho)\setminus \{y\}}t^2f(t)^4 \mathrm{d} \mu(t) \geq w,
\end{align*}
where the term $R_1$ can be bounded as follows
\[ 0 \leq R_1 \leq \frac{\varepsilon}{y^2}. \]
In particular, we have
\begin{align*}
 \left( \int f(t)^2 \mathrm{d} \mu (t) \right)^2 - \int t^2 f(t)^4 \mathrm{d} \mu (t) \leq (w+R_1)^2-w  \leq w^2+w\left(2\frac{\varepsilon}{y^2}-1\right)+\frac{\varepsilon^2}{y^4} = E(\varepsilon) < 0
\end{align*}
By the Weierstrass approximation theorem, we can uniformly approximate $f$ by a polynomial $h \in \R[t]$ and preserve this inequality. Evaluating $\varphi \in \mathcal{K}_T$ on the symmetric sum of squares $e_2(X_1^2 h(X_1)^2,\ldots) \in A$ (see Lemma \ref{lem: e_2(..) in A}) yields a negative evaluation which contradicts our assumption.

    \item Write $\mu = \sum_{i=1}^\infty a_i \delta_{x_i}$ with $a_i > 0$ for all $i \in \N$ and with moment sequence $(m_j)$. Suppose that $0 \neq x \in [-1,1]$ with $N_x := \frac{\mu(x)}{x^2} \not\in \N$. Let $k \geq 2$ be the integer with $N_x\in (k-2,k-1)$. Observe that $\mu(x) > 0$, so $x$ is an isolated point in the support of $\mu$ by (4). 
    Write $e_k=g(p_1,\ldots,p_k)$ for some polynomial $g \in \R[Z_1,\ldots,Z_k]$. By Lemma \ref{lem: Luca's formula} we have $g(t,\ldots,t)=\frac{1}{k!}\prod_{i=0}^{k-1}(t-i)$ which implies $g(N_x,\ldots,N_x)<0$. Because $g$ is continuous, there exists $\varepsilon > 0$ such that $g(z_1,\ldots,z_k) < 0$ for all $(z_1,\ldots,z_k) \in \R^k$ with $|z_i-N_x| < \varepsilon$ for all $1 \leq i \leq k$.
    Since $x$ is an isolated point in the support of $\mu$, there exists $\rho > 0$ with $\operatorname{supp}(\mu) \cap (x-\rho,x+\rho)=\{x\}$. We define the continuous function
    \[ f : [-1,1] \to \R, t \mapsto \max\left\{0,\frac{1}{|x|}\cdot \left(1-\frac{|t-x|}{\rho} \right)\right\}.\]
    Thus, $f$ is piecewise-linear and equal to $0$ outside the interval $(x-\rho,x+\rho)$. For every integer $m \geq 1$ we obtain
    \[ \int t^{2m-2}f(t)^{2m} \mathrm{d} \mu (t) = x^{2m-2}f(x)^{2m}\mu(x)=\frac{x^{2m-2}}{x^{2m}}\mu(x) = \frac{\mu(x)}{x^2} = N_x. \]
    By the Weierstrass approximation theorem, there exists a polynomial $h \in \R[t]$ such that
    \[ \left| \int t^{2m-2}h(t)^{2m} \mathrm{d} \mu(t) - N_x \right| < \varepsilon  \qquad \text{for all } 1 \leq m \leq k.\]
    For $1 \leq m \leq k$, we define the symmetric function $p_{m,h}:=p_m(X_1^2h(X_1)^2,\ldots) \in A \cap \SSigma \subset T$ for which we obtain by (2)
    \begin{align*}
        \varphi (p_{m,h})  &= \varphi \left(\sum_{i=1}^\infty X_i^{2m}h(X_i)^{2m} \right) = \int t^{2m-2}h(t)^{2m} \mathrm{d} \mu (t). 
    \end{align*}
    In particular we have $|\varphi (p_{m,h})-N_x| < \varepsilon$ for all $1 \leq m \leq k$. 
    So we have
    \begin{align*}
    \varphi \left( e_k\left(X_1^2h\left(X_1\right)^2,\dots \right)\right) & = \varphi \left(g(p_{1,h},\cdots,p_{k,h})\right) = g\left(\varphi(p_{1,h}),\dots,\varphi (p_{k,h})\right) < 0
    \end{align*}
    which is a contradiction to $\varphi \in \mathcal{K}_T$.
    \item For $n \geq 1$, consider the sequence 
    \[y^{(n)} = \Bigg( \underbrace{\frac{\sqrt{c}}{\sqrt{n}},\ldots,\frac{\sqrt{c}}{\sqrt{n}}}_{\# = n},\underbrace{x_1,\ldots,x_1}_{\# = N_1},\underbrace{x_2,\ldots,x_2}_{\# = N_2},\ldots\Bigg) \in \ell^2 \]
    for which the associated sequence of measures $\mu_{y^{(n)}}$ converges pointwise to $\mu$.
    Because $\R^\infty$ is dense in $\ell^2$ the statement follows.
    \end{enumerate}
\end{proof}

The following result is an immediate consequence of Lemma \ref{lem: K_T classification}. We state it explicitly again, because of its importance in our proof of Theorem \ref{thm: dimension-independent reznick} and we think that it could be of independent interest.
\begin{corollary}\label{cor:K_T-description}
Let $A=\R[p_2,p_3,\ldots]$, $I \subset A$ the ideal generated by $1-p_2$ and $T= (\SSigma \cap A) + I$. Then
    \[\{ (\varphi(p_2),\varphi(p_3),\ldots) ~|~ \varphi \in \mathcal{K}_T \}=\operatorname{cl} \{ (p_2(x),p_3(x),\dots) ~|~ x\in \R^\infty, ~ \Vert x\Vert_2=1 \},\]
    where the closure is taken with respect to the product topology on $\R^{\N}$.
\end{corollary}

In particular, combining Proposition \ref{prop: orbit space} and Corollary \ref{cor:K_T-description} we obtain that the closure of the orbit space of the slice of $\ell^2$ at $p_2=1$ with respect to the group $\S_\infty$ modulo redundant $0$ entries can be identified with the character space $\mathcal{K}_T$ of non-negative ring homomorphisms on $T$. 
Moreover, the character space $\mathcal{K}_T$ corresponds to the Thoma simplex which characterizes the characters of all irreducible representations of the small infinite symmetric group \cite{Thoma64}. 

The following results will be used in Section \ref{sec: generalizations reznick} to prove generalizations of Theorem \ref{thm: dimension-independent reznick} concerning a non-homogeneous variant and homogeneous symmetric functions that involve $p_1$. Therefore, we consider the ring 
\[ B:=\R[X_0,p_2,p_3,\ldots] \subset \R[[X_0,X_1,X_2,\ldots]]\]
and the set of \emph{any-dimensional} sums of squares in $B$, i.e., those elements $f \in B$ whose truncations to all finite numbers of variables are sums of squares. We denote this set by $\widetilde \SSigma.$

\begin{corollary}\label{cor:K_T-description non-homogeneous}
Let $B=\R[X_0,p_2,p_3,\ldots]$, $\widetilde I \subset B$ the ideal generated by $1-p_2-X_0^2$ and $\widetilde T= (\widetilde{\SSigma} \cap B) + \widetilde I$. Then
    \[\{ (\varphi(X_0),\varphi(p_2),\varphi(p_3),\ldots) ~|~ \varphi \in \mathcal{K}_{\widetilde T} \}=\operatorname{cl} \{ (x_0,p_2(x),p_3(x),\dots) ~|~ (x_0,x)\in \R \times \R^\infty, ~ x_0^2+p_2(x)=1 \},\]
    where the closure is taken with respect to the product topology on $\R^{\N}$.
\end{corollary}
\begin{proof}
It is clear that the point evaluations are valid elements in the character space of $\widetilde T$. To prove the reverse inclusion let $\varphi$ be in $\mathcal{K}_{\widetilde T}$ and set $a:=\varphi(X_0)$.
\\
\underline{Case 1:} Suppose that $a^2=1$. Then $\varphi(p_2)=0$, because $1-p_2-X_0^2 \in \widetilde T$ and $p_2\in \widetilde T$. From the proof of Archimedeanity of $\widetilde T$, we obtain $p_2\pm p_k \in \widetilde T$, so $\varphi(p_k)=0$. So $\varphi$ is realized by the point $(a,0,0,0,\dots)$.
\\
\underline{Case 2:} Suppose that $a^2<1$. Then we define a homomorphism
$\widetilde \varphi: A\to \R$ by 
\[\widetilde \varphi(p_k)=(1-a^2)^{-k/2}\varphi(p_k)=\varphi(p_k((1-a^2)^{-1/2}X_1,(1-a^2)^{-1/2}X_2,\dots)).\]
Then $\widetilde \varphi \in \mathcal{K}_T$. Therefore, applying Corollary \ref{cor:K_T-description} and rescaling yields the claim.
\end{proof}
\begin{corollary}\label{cor:K_T-description with p_1}
    Let $A=\R[p_2,p_3,\ldots]$, $I \subset A$ the ideal generated by $1-p_2$ and $T= (\SSigma \cap A) + I$. Then
    \[\R \times \{(\varphi(p_2),\varphi(p_3),\ldots) ~|~ \varphi \in \mathcal{K}_T \}=\operatorname{cl} \{ (p_1(x),p_2(x),p_3(x),\dots) ~|~ x\in \R^\infty, ~ p_2(x)=1\},\]
    where the closure is taken with respect to the product topology on $\R^{\N}$.
\end{corollary}
\begin{proof}
Since $p_1$ can take any value independently of $p_2,p_3,\ldots$ (if we allow taking the closure) the containment of the set on the right-hand side in the set on the left-hand side follows directly.
It remains to prove the remaining inclusion. Let $a\in \R$ and $\varphi \in \mathcal{K}_T$. By Corollary \ref{cor:K_T-description}, there are $x^{(n)}\in \R^\infty$ with $p_2(x^{(n)})=1$ and $p_k(x^{(n)})$ converges to $\varphi(p_k)$ for each integer $k \geq 3$.
    \\
    Define $b_n:= p_1(x^{(n)})$ and $c_n:=a-b_n$ and choose $M_n\in \N$ large enough that \[\left| \frac{c_n^{k}}{M_n^{k-1}} \right|\leq \frac{1}{n} \quad \text{for all $k\geq 2$}.\]
    This is possible, since the inequality $\max\left\{\left| \frac{c_n}{M_n} \right|,\left| \frac{c_n^2}{M_n} \right|  \right\}\leq 1$ suffices. Consider now
    \[y^{(n)}:=\Big( \underbrace{\frac{c_n}{M_n},\dots,\frac{c_n}{M_n}}_{\#=M_n}, x_1^{(n)},x_2^{(n)},x_3^{(n)},\dots\Big).\]
    We have $p_1(y^{(n)})=c_n+b_n=a$ and $p_k(y^{(n)})-p_k(x^{(n)})\xrightarrow{n\to \infty} 0$ for each $k\geq 2$. Define the normalized sequence
    \[z^{(n)}:=\frac{1}{\sqrt{p_2(y^{(n)})}}y^{(n)}=\frac{1}{\sqrt{1+\frac{c_n^2}{M_n}}}y^{(n)}\]
    and observe 
    \[p_1(z^{(n)})=\frac{a}{\sqrt{1+\frac{c_n^2}{M_n}}}\xrightarrow{n\to \infty}a, ~ p_2(z^{(n)})=1 \text{ and } p_k(z^{(n)})=\frac{p_k(y^{(n)})}{\left(1+\frac{c_n^2}{M_n}\right)^{k/2}}\xrightarrow{n\to \infty}\varphi(p_k) \text{ (for $k\geq 3$)},
    \]
    i.e., $(a,\varphi(p_2),\varphi(p_3),\dots)\in \operatorname{cl} \{ (p_1(x),p_2(x),p_3(x),\dots) ~|~ x\in \R^\infty, p_2(x)=1\}$.
\end{proof}

\section{Two generalizations of Theorem~\ref{thm: dimension-independent reznick}}\label{sec: generalizations reznick}
In this section we generalize the any-dimensional Reznick Positivstellensatz (Theorem~\ref{thm: dimension-independent reznick}) to non-homogeneous symmetric functions in $\R[p_2,p_3,\ldots]$ and symmetric functions in $\Lambda=\R[p_1,p_2,\ldots]$.
\subsection{Reznick's Positivstellensatz for non-homogeneous symmetric functions}
\begin{theorem}\label{thm: dimension-independent reznick non-homogeneous}
    Let $\varepsilon > 0$ and $f\in \R[p_2,p_3,\dots]$ be a symmetric function such that 
    \[
    f^h(x_0,x)\geq \varepsilon \text{ for all } (x_0,x)\in \R\times \R^\infty \text{ with } x_0^2+p_2(x)=1,
    \]
    where $f^h$ denotes the homogenization of $f$ with respect to $X_0$.
    Then there exists $k\in \N_0$ such that $(1+p_2)^k f$ is any-dimensional sos.
\end{theorem}
\begin{proof}
Let $B:=\R[X_0,p_2,p_3,p_4,\dots]$ and $\widetilde{\SSigma}\subseteq B$ be the set of any-dimensional symmetric sos in $B$, i.e., $f\in \widetilde{\SSigma}$ if and only if $f^{(n)}:=f(X_0,X_1,\dots,X_n,0,0,\dots)$ is a sos for all  $n\in \N$. Consider now the preordering
$\widetilde{T}:=\widetilde{\SSigma}+\widetilde I$, where $\widetilde I$ is the ideal in $B$ generated by $1-X_0^2-p_2$.
The preordering $\widetilde T$ is Archimedean, because $1-X_0^2=1-X_0^2-p_2+p_2 \in \widetilde T$, so $1\pm X_0 = \frac{1}{2}((1-X_0^2)+(X_0\pm 1)^2) \in \widetilde T$ and $T\subseteq \widetilde T$ (see Lemma \ref{lem: T is Archimedean}).
\\
By Corollary \ref{cor:K_T-description non-homogeneous}, we have that $f^h>0$ on $\mathcal{K}_{\widetilde T}$, so we can use the Representation Theorem \ref{thm: representation theorem} and conclude that $f^h\in \widetilde T$. Analogously to the proof of Theorem \ref{thm: dimension-independent reznick}, we substitute $\frac{X_i}{\sqrt{X_0^2+p_2}}$ for each $X_i$ and clear denominators to get that
\[(X_0^2+p_2)^kf^h \in \widetilde{\SSigma}.\]
We conclude that the dehomogenization $(1+p_2)^kf$ is an any-dimensional sos.
\end{proof}

\subsection{A Reznick Positivstellensatz including \texorpdfstring{$p_1$}{p1}}

We will need the following cylinder Positivstellensatz, which is a generalization by Schmüdgen and Schmötz \cite[Theorem~11.1]{schmudgen2024positivstellensatze} of a cylinder Positivstellensatz by Powers \cite{powers2004positive}.
\begin{theorem}[Cylinder Positivstellensatz]\label{thm: cylinder positivstellensatz}
Let $\mathcal{A}$ be a finitely generated commutative unital $\R$-algebra and $S$ an Archimedean semiring of $\mathcal{A}$. Suppose that $J$ is an unbounded closed subset of $\R$. Let
$f = \sum_{i=0}^{2r} f_i t^i \in \mathcal{A}[t]$ with coefficients $f_0, \dots, f_{2r} \in \mathcal{A}, ~ r \in \N_0$. Then the following are equivalent:
\begin{enumerate}
    \item $\varphi(f(a)) > 0$ for all $(a,\varphi) \in J \times \mathcal{K}_S$ and $\varphi(f_{2r}) > 0$ for all $\varphi \in \mathcal{K}_S$.
    \item There exists $\delta >0$ such that $f \in \delta(\mathbb{1} + t^2)^r + S \otimes \operatorname{Pos} (J)$, where \[S \otimes \operatorname{Pos} (J):=\left\{ \sum_{i=1}^m s_ir_i ~\middle|~ m \in \N, s_i\in S,~ r_i\in \R[t] \text{ such that } r_i \geq 0 \text{ on } J \right\}.\]
\end{enumerate}
\end{theorem}

In order to use Theorem \ref{thm: cylinder positivstellensatz}, we need to restrict to finitely many power sums, so we get a finitely generated $\R$-algebra $A_{N}:=\R[p_2,p_3,\dots,p_N]$. The following lemma shows that the character space of the restricted preordering is the restriction of the character space of the full preordering.

\begin{lemma}\label{lem: K_T projection}
    Let $T$ be an Archimedean preordering of $A:=\R[p_2,p_3,\dots]$. Denote by $A_{N}:=\R[p_2,p_3,\dots,p_N]$ and $T_N:=T\cap A_N$ the contraction to the first $N-1$ generators. Then the characters of $T_N$ in $A_N$ are exactly the projections of characters of $T$, i.e., we have
    \[\mathcal{K}_{T_N} = \pi_N(\mathcal{K}_T):=\left\{ \varphi|_{A_N} ~\middle|~ \varphi \in \mathcal{K}_T \right\}.\]
\end{lemma}
\begin{proof}
    The inclusion $\pi_N(\mathcal{K}_T)\subseteq \mathcal{K}_{T_N}$ is trivial. The preordering $T_N$ is obviously Archimedean and the character space $\mathcal{K}_T$ and therefore also $\pi_N(\mathcal{K}_T)$ are compact.
    \\
    Suppose $\mathcal{K}_{T_N}\not\subseteq \pi_N(\mathcal{K}_T)$, i.e., there is $\varphi\in \mathcal{K}_{T_N}\setminus \pi_N(\mathcal{K}_T)$. Consider the polynomial
    \[d=\sum_{i=2}^N \left(p_i - \varphi\left(p_i\right)\right)^2.\]
    Since $\varphi\notin \pi_N(\mathcal{K}_T)$ and $\pi_N(\mathcal{K}_T)$ is compact, we have
    \[m:=\min \left\{ \psi(d) ~\middle|~ \psi \in \pi_N(\mathcal{K}_T) \right\} >0.\]
    Now the polynomial $f:=d-\frac{m}{2}$ is positive on $\mathcal{K}_T$ and therefore $f\in T$ by the Representation Theorem \ref{thm: representation theorem}, but $f\in A_N$ by construction which implies $f\in T_N$. So $\varphi(f)\geq 0$ which is a contradiction to $\varphi(f)=-\frac{m}{2}<0$.
\end{proof}

The following theorem is a generalization of Theorem \ref{thm: dimension-independent reznick} to symmetric functions including $p_1$.

\begin{theorem}\label{thm: dimension-independent reznick with p_1}
Let $d\in\N$, $\varepsilon>0$, and let $f\in \Lambda$
be homogeneous of degree $2d$. Write $f$ as a polynomial in $p_1$, i.e.
\[
    f=\sum_{j=0}^{2r} f_j p_1^j
    \qquad \text{for some} \quad r\in \N_0, ~f_0,\dots,f_{2r}\in A_{2d}:=\R[p_2,p_3,\dots,p_{2d}],
\]
where $f_{2r}\neq 0$. Assume that
\begin{enumerate}
    \item $f(x)\geq \varepsilon$ for all $x\in\R^\infty$ with $p_2(x)=1$.
    \item $f_{2r}(x) \geq \varepsilon$ for all $x\in\R^\infty$ with $p_2(x)=1$.
\end{enumerate}
Then there exists $k\in\N_0$ such that
\[
    p_2^k f\in\SSigma .
\]
\end{theorem}
Note that any $f =\sum_{j=0}^{2r+1} f_j p_1^j \in \Lambda$ with $f_j\in A_{2d}:=\R[p_2,p_3,\dots,p_{2d}]$ and $f_{2r+1}\neq 0$ cannot be non-negative. This follows from the fact that $p_1$ can attain any value independently of $p_2,\ldots,p_{2d}$. 
\begin{proof}
    Let $T:=\SSigma\cap A+I$, where $I$ is the ideal generated by $1-p_2$ in $A:=\R[p_2,p_3,\dots]$ and let $T_{2d}:=T\cap A_{2d}$. Consider $F=\sum_{j=0}^{2r} f_j t^j\in A_{2d}[t]$. Condition (1) implies that
    $F(a)>0$ on $\mathcal{K}_{T_{2d}}$ for all $a\in \R$ using Corollary \ref{cor:K_T-description with p_1} and condition (2) implies that $f_{2r}>0$ on $\mathcal{K}_{T_{2d}}$ using Corollary \ref{cor:K_T-description} and Lemma \ref{lem: K_T projection}. Now we can use the cylinder Positivstellensatz (Theorem~\ref{thm: cylinder positivstellensatz}) with $\mathcal{A}=A_{2d}$, $S=T_{2d}$ and $J=\R$ to obtain $\delta >0$ such that
    \[F\in \delta(1 + t^2)^r + T_{2d} \otimes \operatorname{Pos} (\R).\]
    Since $\operatorname{Pos} (\R)=\{ q_1^2 + q_2^2 ~|~ q_1,q_2\in\R[t]\}$, we get
    \[f=F(p_1)\in \delta(1 + p_1^2)^r + T_{2d} \otimes \{ q_1^2 + q_2^2 ~|~ q_1,q_2\in\R[p_1]\}\subseteq \SSigma + (1-p_2)\Lambda.\]
        Analogously to the proof of Theorem \ref{thm: dimension-independent reznick}, we substitute $\frac{X_i}{\sqrt{p_2}}$ for each $X_i$ and clear denominators to get that $p_2^kf \in \SSigma$.
\end{proof}
Note that if $f$ does not involve $p_1$, then Theorem~\ref{thm: dimension-independent reznick with p_1} gives us exactly Theorem~\ref{thm: dimension-independent reznick}. 

\section{Bad points of symmetric functions}
In this section we construct an even homogeneous symmetric function $f\in \R[p_2,p_3,\dots,p_{2d}]$ that is non-negative on $\R^\infty$ and $p_2^kf$ is not any-dimensional sos for any positive integer $k$. The example demonstrates that one cannot replace $\varepsilon$ by $0$ in Corollary \ref{cor:infinite Polya} or Theorem \ref{thm: dimension-independent reznick}.
\\
To this end, we introduce the notion of a bad point.

\begin{definition}
Let $f\in  \R[X_1,\dots,X_n]$ be positive semidefinite and
$P\in  \R^n$. We say that $P$ is a \emph{bad point} of $f$ if
for every identity
\[
f=\sum_{i=1}^r \left(\frac{r_i}{s}\right)^2
\]
with $r_i,s\in \R[X_1,\dots,X_n]$, one has $s(P)=0$.
\end{definition}

Having a bad point $P$ is the same as saying that the polynomial is not a sos in the corresponding local ring
\[\R[X_1,\dots,X_n]_P:= \left\{ \frac{r}{s} \in \R(X_1,\dots,X_n) ~\middle| ~ s(P)\neq 0\right\}.\]
\begin{lemma}
Let $P\in \R^n$ and let $\R[X_1,\dots,X_n]_P$ denote the corresponding local ring. For a non-negative polynomial $f$, the following are equivalent:
\[
P \text{ is a bad point of } f
\quad\Longleftrightarrow\quad
f\notin \sum \R[X_1,\dots,X_n]_P^2.
\]
\end{lemma}
\begin{proof}
If $f=\sum_i \left(\frac{r_i}{s}\right)^2$ with $s(P)\neq 0$, then each $\frac{r_i}{s}$ belongs to the local ring $\R[X_1,\dots,X_n]_P$, so $f\in \sum \R[X_1,\dots,X_n]_P^2$. Conversely, if
\[
f=\sum_i g_i^2 \qquad (g_i\in \R[X_1,\dots,X_n]_P),
\]
write $g_i=\frac{r_i}{s_i}$ with $s_i(P)\neq 0$. After clearing denominators, this
yields a rational sum of squares representation of $f$ with a denominator that
does not vanish at $P$.
\end{proof}

If a polynomial has no bad point at the origin, then its lowest degree non-zero homogeneous part is a sum of squares. We write
\[\R[X_1,\dots,X_n]_0 \]
for the local ring at the origin.

\begin{lemma}\label{lem: sos_local}
Let $f\in \R[X_1,\dots,X_n]$ and write
\[f=f_k+f_{k+1}+f_{k+2}+\cdots+f_d\]
where each $f_i\neq 0$ is homogeneous of degree $i$.
If $f$ is a sum of squares in $\R[X_1,\dots,X_n]_0$, then $f_k$ is a sum of squares of
homogeneous polynomials in $\R[X_1,\dots,X_n]$.

In particular, if the lowest degree non-zero homogeneous part $f_k$ is not a sum of
squares, then $f$ is not a sum of squares in $\R[X_1,\dots,X_n]_0$.
\end{lemma}
\begin{proof}
Suppose that $f$ is a sum of squares in $\R[X_1,\dots,X_n]_0$, i.e.,
\[
f=\sum_{j=1}^N \frac{r_j^2}{s^2}
\]
for some $r_j,s\in \R[X_1,\dots,X_n]$ with $s(0)\neq 0$. Thus
\[s^2f=\sum_{j=1}^N r_j^2\]
which implies that $s^2f$ is a sum of squares in $\R[X_1,\ldots,X_n]$. Therefore its lowest degree homogeneous part $s(0)^2f_k$ and in particular $f_k$ must be sums of squares in $\R[X_1,\dots,X_n]$.
\end{proof}

The lemma gives a construction for polynomials with bad points due to Delzell \cite[Chapter V]{delzell1980constructive}. The following concrete example can be found in \cite{kaltofen2012exact}:

\begin{example}
    The Delzell polynomial is defined as
    \[D=X_1^4 X_2^2 X_4^2 + X_2^4 X_3^2 X_4^2 + X_1^2 X_3^4 X_4^2 - 3X_1^2 X_2^2 X_3^2 X_4^2 + X_3^8 \in \R[X_1,X_2,X_3,X_4].\]
    We want to show that $P=(0,0,0,1)$ is a bad point of $D$.
    If we dehomogenize with $X_4=1$ we get
    \[d=D(X_1,X_2,X_3,1)=M+X_3^8,\]
    where $M=X_1^4 X_2^2 + X_2^4 X_3^2 + X_1^2 X_3^4 - 3X_1^2 X_2^2 X_3^2$ is psd by the AM–GM inequality and not a sos (similar to the Motzkin polynomial \cite{motzkin1967arithmetic}). So $d$ is not sos in $\R[X_1,X_2,X_3]_0$ by Lemma \ref{lem: sos_local} and therefore $D$ is not sos in $\R[X_1,X_2,X_3,X_4]_P$. 
    So $D$ is positive semidefinite but has a bad point at $(0,0,0,1)$. Therefore
    \[\left(\sum_{i=1}^4 X_i^2\right)^kD\]
    is not a sum of squares for any $k\in\N$.
\end{example}

    We can use a similar construction to get symmetric polynomials with bad points. Consider the Robinson polynomial \cite{robinson1973some}
    \[R=X_1^6+X_2^6+X_3^6-\left(X_1^4X_2^2+X_1^2X_2^4+X_1^4X_3^2+X_1^2X_3^4+X_2^4X_3^2+X_2^2X_3^4\right)+3X_1^2X_2^2X_3^2\]
    which is psd and not a sos.

\begin{example}
For $m\ge 4$, define
\[
R_m(X_1,X_2,X_3,X_4)
=
\sum_{i=1}^4 X_i^{2m}R\left(X_1,\dots,\widehat{X_i},\dots,X_4\right),
\]
where $\widehat{X_i}$ means that the variable $X_i$ is omitted, and
$R$ denotes the Robinson polynomial in the remaining three variables.
Then $R_m$ is symmetric, psd and $P=(0,0,0,1)$ is a bad point of $R_m$. In particular,
    \[\left(\sum_{i=1}^4 X_i^2\right)^kR_m\]
    is not a sum of squares for any $k\in\N$.
\end{example}
\begin{proof}
The dehomogenization with $X_4=1$ is
\[r_m(X_1,X_2,X_3):=R_m(X_1,X_2,X_3,1)=R+X_1^{2m} R(1,X_2,X_3)+X_2^{2m}R(X_1,1,X_3)+X_3^{2m}R(X_1,X_2,1).\]
Since $m\geq 4$, the lowest degree non-zero homogeneous part is equal to the Robinson polynomial $R$. Hence $R_m$ is not a sum of squares in the local ring $\R[X_1,\dots,X_4]_P$ and therefore $P$ is a bad point of $R_m$.
\end{proof}

We can also extend this construction to get symmetric functions whose truncations to sufficiently many variables have bad points:

\begin{example}\label{ex: bad point reznick}
For $m\geq 6$ define
\[
A_m=\sum_{i=1}^\infty X_i^{2m}A\left(\widehat{X_i}\right),
\]
where $\widehat{X_i}$ means that the variable $X_i$ is omitted, and
\[A=\frac{1}{18}p_2^5+3p_8p_2+6p_6p_4-3p_6p_2^2\]
denotes the decic $A$ (from \cite[Proposition~5.27]{acevedo2025symmetric}) in the remaining variables.
Then $A_m$ is symmetric, any-dimensional psd and $P^{(n)}=(1,0,0,0,\dots,0)\in \R^n$ is a bad point of $A_m^{(n)}$ for all sufficiently large $n$. In particular, for sufficiently large $n$,
    \[\left(\sum_{i=1}^n X_i^2\right)^kA_m^{(n)}\]
    is not a sum of squares for any $k\in\N$.
Furthermore, $A_m$ can be expressed in the power sums by
\begin{align*}
A_m &
=Ap_{2m}
+\left(-\frac{5}{18}p_2^4+6p_2p_6-3p_8\right)p_{2m+2}
+\left(\frac{5}{9}p_2^3-9p_6\right)p_{2m+4}
+\left(\frac{22}{9}p_2^2-6p_4\right)p_{2m+6} \\
&
-\frac{157}{18}p_2p_{2m+8}
+\frac{215}{18}p_{2m+10}.    
\end{align*}

\end{example}
\begin{proof}
The dehomogenization with $X_1=1$ is
\[a_m(X_2,X_3,X_4,\dots):=A(X_2,X_3,X_4,\dots)+\sum_{i=2}^\infty X_i^{2m}A\left(\widehat{X_i}\right).\]
Since $m\geq 6$, the lowest degree non-zero homogeneous part is equal to $A$, which is psd and not sos for any $n$ large enough. Hence $a_m^{(n)}$ is not sos in $\R[X_1,\dots,X_n]_0$ by Lemma \ref{lem: sos_local} and therefore $A_m^{(n)}$ is not a sum of squares in the local ring $\R[X_1,\dots,X_n]_{P^{(n)}}$ for any sufficiently large $n$. So $A_m$ is positive semidefinite but has a bad point at $P^{(n)}$.
\\
To get the expression of $A_m$ in the power sums, we consider
\[B = \frac{1}{18}Z_2^5+3Z_8Z_2+6Z_6Z_4-3Z_6Z_2^2 \in \R[Z_2,Z_4,Z_6,Z_8].\]
We obtain
\begin{align*}
A\left(\widehat{X_i}\right)&=B\left(p_2-X_i^2,p_4-X_i^4,p_6-X_i^6,p_8-X_i^8\right) \\
&=A + \left(- \frac{5}{18} p_2^4 + 6 p_2 p_6 - 3 p_8\right) X_i^2 + \left(\frac{5}{9} p_2^3- 9 p_6\right)X_i^4
 +\left(\frac{22}{9} p_2^2- 6 p_4\right) X_i^6 - \frac{157}{18} p_2 X_i^8 + \frac{215}{18} X_i^{10}
\end{align*}
from which one obtains the desired expression of $A_m$.
\end{proof}

\begin{remark}
Example \ref{ex: bad point reznick} demonstrates that the condition $f  \geq \varepsilon > 0$ on the sphere in Corollary \ref{cor:infinite Polya} or Theorem \ref{thm: dimension-independent reznick} cannot be relaxed to $f$ being non-negative. This leaves open the case of $f$ being positive definite but having infimum $0$.
\end{remark}
\section{Positivstellensätze for normalized symmetric functions}
\label{sec:normalized}

In this section we consider \emph{power means} instead of power sums. It was already observed in \cite{acevedo2025power,blekherman2021symmetric} that the set of non-negative normalized symmetric functions is semialgebraic. Moreover, this normalized setting is essentially equivalent to the study of univariate polynomial inequalities in moments of probability measures on the real line, or normalized univariate trace polynomials. These settings have been investigated (in the multivariate case) independently in \cite{klep2022optimization, klep2026sums, klep2021}. The frameworks developed there are more general than normalized symmetric functions. For instance, they address multivariate moment inequalities in \cite{klep2026sums} and trace evaluations on non-commuting variables using operator-algebraic methods and the GNS construction in \cite{klep2021}.
Here, we present an alternative and elementary proof for the univariate setting which essentially follows from \cite{acevedo2025power}.
Our proof demonstrates that for normalized symmetric functions (equivalently, univariate pure trace polynomials or pure moment polynomials) the situation is actually simple. The Positivstellensätze follow directly from classical real algebraic geometry on Hankel spectrahedra.

\begin{definition}
By scaling the $k$-th power sum in $n$ variables $p_k^{(n)}$ by the factor $\frac{1}{n}$, we call $\frac{p_k^{(n)}}{n}$ the \emph{$k$-th power mean} in $n$ variables. Let $\Theta := \R[\mathfrak{p}_1,\mathfrak{p}_2,\ldots]$ denote the ring of \emph{normalized symmetric functions} in the power sum basis, where we consider $\mathfrak{p}_1,\mathfrak{p}_2,\ldots$ as homogeneous algebraically independent variables of degrees $1,2,\dots$. We identify $f=g(\mathfrak{p}_1,\ldots,\mathfrak{p}_d) \in \Theta$ with a sequence of symmetric polynomials $f^{(n)}=g\left(\frac{p_1^{(n)}}{n},\ldots,\frac{p_d^{(n)}}{n}\right)$. 
We call $f$ \emph{non-negative} (respectively \emph{sos}) if $f^{(n)} \in \mathcal{SP}_n$ (respectively $f^{(n)}\in \SSigma_n$) for all $n \in \N$. 
\end{definition}
The graded rings $\Lambda$ and $\Theta$ are isomorphic, but the evaluations of elements in these rings differ.

\begin{example}
    Consider $g(Z_1,Z_2,Z_3,Z_4):= Z_4-Z_2^2$. Replacing $Z_i$ by $p_i^{(n)}$ gives the symmetric polynomial
    \[g(p_1^{(n)},\dots,p_4^{(n)})=-2e_2^{(n)}(X_1^2,\dots,X_n^2)\]
    and therefore the symmetric function $g(p_1,\ldots,p_4) \in \Lambda$ is not non-negative.
    In power means we obtain
    \[g\left(\frac{p_1^{(n)}}{n},\dots,\frac{p_4^{(n)}}{n} \right)=\frac{1}{n^2} \sum_{1\leq i<j\leq n} (X_i^2-X_j^2)^2 \]
    and therefore $g(\mathfrak{p}_1,\dots,\mathfrak{p}_4) \in \Theta$ is a sum of squares.
\end{example}

We recall from \cite[Proposition~2.6]{blekherman2021symmetric} that $f^{(d\cdot n)}$ non-negative implies $f^{(n)}$ non-negative. Analogously, if $f^{(d\cdot n)}$ is sos then $f^{(n)}$ is sos. This follows immediately from the \emph{duplication map}
\begin{align}\label{eq: duplication}
\frac{p_i^{(n)}(X_1,\ldots,X_n)}{n} = \frac{p_i^{(d \cdot n)}(\overbrace{X_1,\ldots,X_1}^{\# = d},\ldots,\overbrace{X_n,\ldots,X_n}^{\# = d})}{dn}.     
\end{align}
We start with the following structural observation from \cite{acevedo2025power}. Let \[\mu_{n,2d}:=\left(\frac{p_1^{(n)}}{n},\ldots,\frac{p_{2d}^{(n)}}{n}\right) : \R^n \to \R^{2d}\]
denote the map consisting of the first $2d$ power means in $n$ variables. 
An element $f=g(\mathfrak{p}_1,\ldots,\mathfrak{p}_{2d}) \in \Theta$ is non-negative if $g(z) \geq 0$ for all points $z$ in the set 
\[\mathcal{H}_{2d} := \operatorname{cl} \bigcup_{n \in \N}  \mu_{n,2d}(\R^n)  .\] 
We can take the closure of the union of the images of the maps $\mu_{n,2d}$, since $g$ being non-negative on the union of $\mu_{n,2d}(\R^n)$ for $n \in \N$ implies by continuity that $g$ is non-negative on the boundary of the union. 
The relation of the set $\mathcal{H}_{2d}$ to univariate probability measures was pointed out in \cite[Section~4]{acevedo2025power}.
In fact, the map $\mu_{n,2d}$ is the moment map up to degree $2d$ of a uniform probability measure on the real line
supported on $n$ points (where the $0$-th moment is dropped). 
For $y \in \R^{2d}$ and $y_0:=1$ we define the Hankel matrix 
\begin{align*}
       H_{2d}(y):= (y_{i+j-2})_{1 \leq i,j \leq d+1}\quad = \quad \begin{pmatrix}
        1 & y_1 & y_2 & \dots & y_{d} \\
        y_1 & y_2 & y_3 & \dots & y_{d+1} \\
        \vdots & \vdots & \vdots & & \vdots \\
        y_d & y_{d+1} & y_{d+2} &  \dots & y_{2d}
    \end{pmatrix}.
\end{align*}
The set of all $y \in \R^{2d}$ for which the matrix $H_{2d}(y)$ is positive semidefinite is a Hankel spectrahedron, i.e., it is defined by a linear matrix inequality where the matrix is a Hankel matrix.
\begin{proposition}[\cite{acevedo2025power}~Proposition~4.1]\label{prop:hankel spectrahedron}
    The set $\mathcal{H}_{2d}$ is the set of all points $y \in \R^{2d}$ for which $H_{2d}(y)$ is positive semidefinite.
\end{proposition}
The Hankel spectrahedron $\mathcal{H}_{2d}$ is semialgebraic but not compact. Moreover, for $f=h(\mathfrak{p}_1,\ldots,\mathfrak{p}_{2d}) \in \Theta$ with $f^{(n)}$ strictly positive for all $n \in \N$, we cannot assume that $h$ is positive on $\mathcal{H}_{2d}$ in general. This is illustrated in the following example.

\begin{example}[\cite{klep2021}~Remark~3.3] \label{example: normalized positive}
For $\alpha =\frac{1}{\sqrt{2}}$ we consider the discrete probability measure $\mu=(1-\alpha)\delta_0+\alpha \delta_1$ and the normalized symmetric function $f = g(\mathfrak{p}_1,\ldots,\mathfrak{p}_4) = \sum_{i=1}^4 (\mathfrak{p}_i-\alpha)^2 \in \Theta$. For $n \geq 1$ the $n$-th moment of $\mu$ equals $(1-\alpha)\cdot 0^n + \alpha \cdot 1^n = \alpha$. Thus, $(\alpha,\alpha,\alpha,\alpha) \in \mathcal{H}_4$ which shows that $g$ is not strictly positive on $\mathcal{H}_4$. On the contrary, the point $(\alpha,\alpha,\alpha,\alpha)$ is never contained in $\mu_{n,4}(\R^n)$ for $n \in \N$. This is because then the first four non-trivial moments of a probability measure $\mu_n = \sum_{i=1}^n \frac{1}{n} \delta_{x_i}$ on $\R$ are $\alpha$, i.e., $\int t^k \mathrm{d}\mu_n(t) = \alpha$ for $1 \leq k \leq 4$. We observe
\[ \sum_{i=1}^n \frac{(x_i^2-x_i)^2}{n} = \int (t^2-t)^2 \mathrm{d}\mu_n(t) = \int t^4 \mathrm{d}\mu_n(t) -2 \int t^3 \mathrm{d}\mu_n(t) + \int t^2 \mathrm{d} \mu_n(t) = \alpha-2\alpha+\alpha=0. \]
Thus each summand $x_i^2-x_i$ must be zero, which implies $x_i \in \{0,1\}$. So $\mu_n = \frac{n-k}{n} \delta_0+ \frac{k}{n} \delta_1$ for some integer $1 \leq k \leq n$. We find
$\alpha = \int t \mathrm{d}\mu_n(t)= \frac{k}{n} $ which is a contradiction since $\alpha$ is irrational.
\end{example}

For a fixed $d \in \N$ let $m_1,\ldots,m_N \in \R[Z_1,\ldots,Z_{2d}]$ denote all principal minors of the matrix $H_{2d}(Z_1,\ldots,Z_{2d})$ whose entries are polynomials. So, we can write the Hankel spectrahedron $\mathcal{H}_{2d}$ as the basic semialgebraic set 
\begin{equation} \label{eq: H_2d as basic semialgebraic set}
\mathcal{H}_{2d} = \left\{ y \in \R^{2d} ~\middle|~ m_1(y)\geq 0,\ldots,m_N(y) \geq 0\right\}.     
\end{equation}
The principal minors evaluated in power means are symmetric sums of squares which will be used to prove the Positivstellensatz. The following lemma is essentially the normalized symmetric function analog of \cite[Lemma~3.6]{klep2026sums}.

\begin{lemma} \label{lem: principal minors become sos}
    Let $1 \leq k \leq d+1$ be an integer and $I \in \binom{[d+1]}{k}$. Let $m_I \in \R[Z_1,\ldots,Z_{2d}]$ denote the principal minor of the principal submatrix $H_{2d,I}=(Z_{i+j-2})_{i,j \in I}$ of the matrix $H_{2d}(Z_1,\ldots,Z_{2d})$. Then $m_I(\mathfrak{p}_1,\ldots,\mathfrak{p}_{2d})$ is sos, i.e., $m_I\left(\frac{p_1^{(n)}}{n},\ldots,\frac{p_{2d}^{(n)}}{n}\right)$ is a sum of squares for all $n \in \N$. 
\end{lemma}
\begin{proof}
    Suppose that $I= \{i_1 < \dots < i_k\}$ and write $H_I:=H_{2d,I}\left(\frac{p_1^{(n)}}{n},\ldots,\frac{p_{2d}^{(n)}}{n}\right)$. Then $\det H_I=m_I\left(\frac{p_1^{(n)}}{n},\ldots,\frac{p_{2d}^{(n)}}{n}\right)$ and we use the Gram decomposition $H_I=Z^\top Z$, where 
    \[ Z = \frac{1}{\sqrt{n}}\begin{pmatrix}
        X_1^{i_1-1} & X_1^{i_2-1} & \dots & X_1^{i_k-1} \\ 
        X_2^{i_1-1} & X_2^{i_2-1} & \dots & X_2^{i_k-1} \\
        \vdots & \vdots & & \vdots \\
        X_n^{i_1-1} & X_n^{i_2-1} & \dots & X_n^{i_k-1} 
    \end{pmatrix} \in \R[X_1,\ldots,X_n]^{n \times k}.\]
    We can apply the Cauchy-Binet formula to calculate $\det H_I$ and obtain
    \[ m_I\left(\frac{p_1^{(n)}}{n},\ldots,\frac{p_{2d}^{(n)}}{n}\right) = \det H_I = \sum_{J \in \binom{[n]}{k}} (\det Z_J)^2,\]
    where $Z_J \in \R[X_1,\ldots,X_n]^{k \times k}$ denotes the submatrix of $Z$ with rows indexed by $J$.
    So, the polynomial $m_I\left(\frac{p_1^{(n)}}{n},\ldots,\frac{p_{2d}^{(n)}}{n}\right)$ is clearly a symmetric sum of squares.
\end{proof}
We point out that the factorization of the matrix $H_I$ is known in the context of polynomial optimization for $n$-atomic measures (see e.g. \cite[Lemma~5.4.]{laurent2008sums}).
Moreover, it is a direct consequence of Proposition \ref{prop:hankel spectrahedron} and Lemma \ref{lem: principal minors become sos} that the (unconstrained) Krivine-Stengle-Positivstellensatz \cite[2.2.1~Positivstellensatz]{marshall2008positive} applies to normalized symmetric functions. We describe the constrained analog in Section \ref{sec: constrained non-negativity}.

\begin{theorem}\label{thm: positivstellensatz normalized setting}
  Let $f=h(\mathfrak{p}_1,\ldots,\mathfrak{p}_{2d}) \in \Theta$.
  \begin{enumerate}
      \item If some $\varepsilon > 0$ exists with $f^{(n)}\geq \varepsilon$ for all $n \in \N$, then there are sums of squares $p,q \in \Theta$ with $pf=1+q$. 
      \item If $f^{(n)} \geq 0$ for all $n \in \N$, then there are sums of squares $p,q \in \Theta$ and an integer $m \geq 0$ with $pf = f^{2m}+q$.
  \end{enumerate}

\end{theorem}
\begin{proof}
We only prove assertion (1). The proof of (2) proceeds analogously. 
By continuity and Proposition \ref{prop:hankel spectrahedron} we have $f^{(n)} \geq \varepsilon$ for all $n \in \N$ if and only if 
\[h(Z_1,\ldots,Z_{2d}) \geq \varepsilon \text{ on }\mathcal{H}_{2d}= \left\{y \in \R^{2d} ~\middle|~ m_1(y) \geq 0,\ldots, m_N(y) \geq 0\right\}\] by equation (\ref{eq: H_2d as basic semialgebraic set}). Let \[T=\left\{ \sum_{e \in \{0,1\}^N} \sigma_e m_1^{e_1}\dots m_N^{e_N} ~\middle|~ \sigma_e \in \sum \R[Z_1,\ldots,Z_{2d}]^2\right\} \]
denote the preordering defined by $m_1,\ldots,m_N \in \R[Z_1,\ldots,Z_{2d}]$. For $e \in \{0,1\}^N$ we write $m^e:=m_1^{e_1}\cdots m_N^{e_N}$. Because $h > 0$ on $\mathcal{H}_{2d}$, the existence of $p,q \in T$ with $ph=1+q$ follows from the Krivine-Stengle-Positivstellensatz for basic semialgebraic sets in the polynomial ring \cite[2.2.1~Positivstellensatz]{marshall2008positive}. So $p = \sum_{e \in \{0,1\}^N}\sigma_e m^e  $ and $q = \sum_{e \in \{0,1\}^N}\tau_e m^e  $ for $\sigma_e,\tau_e \in \sum \R[Z_1,\ldots,Z_{2d}]^{2}$. Substituting $\frac{p_i^{(n)}}{n}$
for $Z_i$ for all $1 \leq i \leq 2d$ we conclude
\begin{align*}
    &~ p(\mu_{n,2d}) h(\mu_{n,2d}) = 1+ q(\mu_{n,2d}) \\
    \iff & \sum_{e \in \{0,1\}^N}\sigma_e(\mu_{n,2d}) m(\mu_{n,2d})^{e}\cdot f^{(n)} = 1 + \sum_{e \in \{0,1\}^N}\tau_e(\mu_{n,2d}) m(\mu_{n,2d})^{e} .
\end{align*}
This is equivalent to $P f =1+ Q$, where $P=p(\mathfrak{p}_1,\ldots,\mathfrak{p}_{2d}),~ Q=q(\mathfrak{p}_1,\ldots,\mathfrak{p}_{2d}) \in \Theta$. Finally, note that $P,Q$ are sos since each $\sigma_e(\mu_{n,2d})$ and $\tau_e(\mu_{n,2d})$ is a symmetric sum of squares and, by Lemma \ref{lem: principal minors become sos}, each $m_i(\mu_{n,2d})$ is a sum of squares.
\end{proof}
Remarkably, the representations of the form $pf=1+q$ and $pf=f^{2m}+q$ with $p,q \in \Theta$ are of such a form that if $f \in \R[\mathfrak{p}_1,\ldots,\mathfrak{p}_{2d}]$ then also $p,q \in \R[\mathfrak{p}_1,\ldots,\mathfrak{p}_{2d}]$ holds.
Another consequence of the theorem is that every non-negative normalized symmetric function has a sum of squares representation in $\operatorname{Quot}(\Theta)$. This is an analog of Artin’s positive solution of Hilbert's 17th problem \cite{artin1927zerlegung} and was already known in the language of moment inequalities (see \cite[Theorem~3.7]{klep2026sums}) and trace inequalities \cite{klep2021}.

\subsection{Constrained non-negativity}\label{sec: constrained non-negativity}
A natural extension lies in the study of non-negativity or positivity on a constrained set, where the constraints are given by a finite number of elements in $\Theta$.
Suppose we have constraints $g_1=h_1(\mathfrak{p}_1,\ldots,\mathfrak{p}_{2d})\geq 0,\ldots,  g_s=h_s(\mathfrak{p}_1,\ldots,\mathfrak{p}_{2d}) \geq 0$ which define the sets 
\[K_g^{(n)} := \left\{ x \in \R^n ~\middle|~ g_i^{(n)}(x_1,\ldots,x_n) \geq 0,1 \leq i \leq s\right\} = \left\{x \in \R^n ~\middle|~  h_i(\mu_{n,2d}(x)) \geq 0, 1 \leq i \leq s\right\}.\] We say that $f \in \Theta$ is \emph{non-negative subject to $g_1 \geq 0, \dots, g_s \geq 0$} if $f^{(n)}$ is non-negative on $K_g^{(n)}$ for all $n \in \N$.
The difficulty in verifying this lies in the observation that in general non-negativity of $f=h(\mathfrak{p}_1,\ldots,\mathfrak{p}_{2d})$ subject to $g_1 \geq 0, \ldots, g_s \geq 0$ is not equivalent to non-negativity of $h$ on the semialgebraic set  
\[ \mathcal{H}_{2d}(g):= \mathcal{H}_{2d} \cap \left\{ y \in \R^{2d} ~\middle|~ h_i(y) \geq 0, 1 \leq i \leq s \right\}.\]
This is because the inclusion 
\[ 
\operatorname{cl} \bigcup_{n \in \N}    \left\{ y \in \mu_{n,2d}(\R^n) ~\middle|~ h_i(y) \geq 0, 1 \leq i \leq s\right\}  \subset \mathcal{H}_{2d}(g)
\]
can be strict. The inclusion is strict if there exists $y \in \mathcal{H}_{2d} \setminus \bigcup_{n \in \N} \mu_{n,2d}(\R^n)$ with $h_j(y)=0$ for some $1 \leq j \leq s$, but for all sequences $(y_m)_m \in \bigcup_{n \in \N} \mu_{n,2d}(\R^n)$ with $y_m \to y$ we have $h_j(y_m)< 0$ for all sufficiently large $m$. Example \ref{example: normalized positive} illustrates this phenomenon. Suppose that $g_1 = -\sum_{i=1}^4\left(\mathfrak{p}_i-\frac{1}{\sqrt{2}}\right)^2 \geq 0$ is the only constraint. Then $K_g^{(n)} = \emptyset$ for all $n \in \N$. However, the set $\mathcal{H}_{2d}(g)=\left\{\left(\frac{1}{\sqrt{2}},\frac{1}{\sqrt{2}},\frac{1}{\sqrt{2}},\frac{1}{\sqrt{2}}\right)\right\}$ is non-empty.
So, there are in general more normalized symmetric functions $f=h(\mathfrak{p}_1,\dots,\mathfrak{p}_{2d})$ that are non-negative subject to $g_1 \geq 0, \dots, g_s \geq 0$ than normalized symmetric functions for which $h$ is non-negative on $\mathcal{H}_{2d}(g)$. If the set $\left\{y \in \mathcal{H}_{2d} ~\middle|~ h_i(y) > 0, 1 \leq i \leq s\right\}$ is dense in $\mathcal{H}_{2d}(g)$, then the two statements are equivalent.

Nevertheless, the proof of Theorem \ref{thm: positivstellensatz normalized setting} extends immediately to the setting where $f \in \Theta$ is bounded below by some $\varepsilon > 0$ on the set $\mathcal{H}_{2d}(g)$, yielding a constrained Krivine-Stengle-Positivstellensatz analog. Moreover, if $\mathcal{H}_{2d}(g)$ is compact we obtain an analog of Schmüdgen's Positivstellensatz. If additionally the quadratic module $M(m_1,\ldots,m_N,h_1,\ldots,h_s) \subset \R[Z_1,\ldots,Z_{2d}]$ is Archimedean, we obtain an analog of Putinar's Positivstellensatz. The equivalent Positivstellensatz for moment inequalities can be found in \cite[Section~4]{klep2026sums}.

\begin{corollary}
    Let $\varepsilon > 0$ and $f=h(\mathfrak{p}_1,\ldots,\mathfrak{p}_{2d}),g_1=h_1(\mathfrak{p}_1,\ldots,\mathfrak{p}_{2d}), \ldots, g_s=h_s(\mathfrak{p}_1,\ldots,\mathfrak{p}_{2d}) \in \Theta $ and suppose that $h \geq \varepsilon$ holds on the set $\mathcal{H}_{2d}(g)$.
    If $\mathcal{H}_{2d}(g)$ is a compact set, then \[f=\sum_{e \in \{0,1\}^s} \sigma_e g_1^{e_1}\cdots g_s^{e_s} \] for some sos $\sigma_e \in \Theta$. Moreover, if the quadratic module $M(m_1,\ldots,m_N,h_1,\ldots,h_s) \subset \R[Z_1,\ldots,Z_{2d}]$ is Archimedean, then 
    \[f= \sigma_0+ \sum_{i=1}^s \sigma_i g_i \] for some sos $\sigma_i \in \Theta$.
\end{corollary}
We omit the proof, because the statement basically follows from Schmüdgen's and Putinar's Positivstellensätze by applying Lemma \ref{lem: principal minors become sos} in the same way as in the proof of Theorem \ref{thm: positivstellensatz normalized setting}.

\section{Conclusion and open questions}

We proved any-dimensional analogs of Pólya's and Reznick's Positivstellensätze for (even) symmetric functions. In both cases, strict positivity uniformly bounded away from zero yields a dimension-independent certificate: in the Pólya-type statement (Corollary~\ref{cor:infinite Polya}) by non-negative coefficients in the even monomial basis, and in the Reznick-type statement (Theorem~\ref{thm: dimension-independent reznick}) by a uniform power of $p_2$ turning the function into an any-dimensional sum of squares. A central ingredient in the proof of Theorem~\ref{thm: dimension-independent reznick} is the description of the relevant character space in terms of moment sequences of discrete probability measures on $[-1,1]$, or equivalently as a closure of the $S_\infty$-orbit space of the unit sphere in $\ell^2$ (Corollary~\ref{cor:K_T-description}).

Moreover, we also studied normalized symmetric functions. In Theorem \ref{thm: positivstellensatz normalized setting}, we reproved a Krivine-Stengle analog for normalized symmetric functions.

Several questions remain open. First, the framework of \cite{levin2025any} gives a general language for any-dimensional optimization problems via representation stability. Our arguments do not use this perspective. It would be interesting to know whether representation stability can give alternative proofs of our Positivstellensätze, yield effective bounds, or strengthen the certificates obtained here.

There are some natural extensions of Theorem~\ref{thm: dimension-independent reznick} that one might try to show. The proof excludes $p_1$, essentially because $p_1$ is unbounded on the set $\{p_2=1\}$ and the corresponding preordering is no longer Archimedean. Using a cylinder Positivstellensatz from \cite{schmudgen2024positivstellensatze} we could prove a generalization for symmetric functions $f \in \Lambda$ that include powers of $p_1$ under an additional mild assumption on the positivity of $f$ (Theorem~\ref{thm: dimension-independent reznick with p_1}). Does an analog of Theorem~\ref{thm: dimension-independent reznick} hold nonetheless without this additional assumption? Can the conclusion be strengthened so that $p_2^k f$ is not only an any-dimensional sum of squares, but a sum of even powers of linear forms, in analogy with Reznick's stronger finite-dimensional statement? Is it possible to construct a bound on $k$ in Theorem~\ref{thm: dimension-independent reznick}? 

Another group of questions concerns the role of strict positivity. In both Corollary~\ref{cor:infinite Polya} and Theorem~\ref{thm: dimension-independent reznick} we assume that $f$ is bounded below by some $\varepsilon>0$ on the $\ell^2$-unit sphere. Example~\ref{ex: bad point reznick} shows that mere non-negativity is not sufficient. It remains unclear whether one can replace the condition $f(x)\geq \varepsilon$ by the weaker condition $f(x)>0$ for all $x\in \ell^2$. Relatedly, one may ask whether deciding any-dimensional positivity or non-negativity of a symmetric function is decidable. The moment-theoretic and real-algebraic techniques used here suggest possible analogs of Hilbert's 17th problem and of Krivine--Stengle type Positivstellensätze for symmetric functions. To show such a result, one has to study the real spectrum instead of the character space.

\subsubsection*{Acknowledgements}
We thank Luca Wellmeier and Markus Schweighofer for fruitful discussions about the problems. We also thank Cordian Riener for pointing out the equivalence of polynomial inequalities in normalized symmetric functions, univariate pure trace polynomials and pure moment polynomials. Furthermore, we thank Claus Scheiderer for helpful insights into the real spectrum.

\bibliographystyle{abbrv}
\bibliography{references}

\end{document}